\documentclass[11pt,letterpapert]{article}
\usepackage{geometry}
\usepackage{graphicx}
\usepackage{amssymb}
\usepackage{epstopdf}
\usepackage{fancyhdr}
\usepackage{amsmath}
\usepackage{amsthm}
\usepackage{hyperref}
\usepackage{makeidx}
\usepackage{blkarray}
\usepackage{mathdesign}
\usepackage{color}
\usepackage{geometry}
\usepackage{soul}
\usepackage{nameref}
\usepackage{stmaryrd}
\usepackage{enumitem} 
\definecolor{purple}{rgb}{.9,0,.9}
\definecolor{green}{rgb}{0,.7,0}

\newtheoremstyle{pstyle}
  {}%                                     
  {}%                                    
  {}
  {}%                            
  {\itshape}%
  {.}%
  { }%
  {}%
\theoremstyle{pstyle}
\newtheorem{proofpart}{Part}

\makeatletter
\let\orgdescriptionlabel\descriptionlabel
\renewcommand*{\descriptionlabel}[1]{%
 \let\orglabel\label
 \let\label\@gobble
 \phantomsection
 \edef\@currentlabel{#1}%
 \let\label\orglabel
 \orgdescriptionlabel{#1}%
}
\makeatother

\graphicspath{{./figures/}}

\newcommand{\Nat}{\mathbb{N}}

\newcommand{\Real}{\mathbb{R}}

\newcommand{\cA}{{\cal A}}\newcommand{\cB}{{\cal B}}\newcommand{\cC}{{\cal C}}
\newcommand{\cD}{{\cal D}}\newcommand{\cE}{{\cal E}}\newcommand{\cF}{{\cal F}}
\newcommand{\cI}{{\cal I}}
\newcommand{\cL}{{\cal L}}
\newcommand{\cM}{{\cal M}}
\newcommand{\cP}{{\cal P}}\newcommand{\cR}{{\cal R}}
\newcommand{\cS}{{\cal S}}\newcommand{\cU}{{\cal U}}
\newcommand{\cV}{{\cal V}}\newcommand{\cX}{{\cal X}}
\newcommand{\cY}{{\cal Y}}

\newcommand{\bc}{{\bf c}}
\newcommand{\bd}{{\bf d}}\newcommand{\bff}{{\bf f}}
\newcommand{\bg}{{\bf g}}\newcommand{\bh}{{\bf h}}

\newcommand{\bm}{{\bf m}}\newcommand{\bn}{{\bf n}}
\newcommand{\bp}{{\bf p}}\newcommand{\bq}{{\bf q}}\newcommand{\br}{{\bf r}}
\newcommand{\bt}{{\bf t}}
\newcommand{\bw}{{\bf w}}\newcommand{\bx}{{\bf x}}

\newcommand{\bH}{{\bf H}}
\newcommand{\bL}{{\bf L}}\newcommand{\bM}{{\bf M}}
\newcommand{\bN}{{\bf N}}
\newcommand{\bQ}{{\bf Q}}
\newcommand{\bT}{{\bf T}}

\newcommand{\tr}{\text{tr}}

\newcommand{\bnu}{\boldsymbol{\nu}}

\newcommand{\bchi}{\boldsymbol{\chi}}

\newcommand{\hhA}{\hat\cA}

\newtheorem{theorem}{Theorem}%[section]

\newtheorem{corollary}[theorem]{Corollary}
\newtheorem{definition}[theorem]{Definition}
\newtheorem{proposition}[theorem]{Proposition}

\newtheorem{remark}[theorem]{Remark}

\numberwithin{equation}{section}

\newcommand{\da}{\,\text{d}a}

\newcommand{\dalpha}{\,\text{d}\alpha}
\newcommand{\dbeta}{\,\text{d}\beta}

\newcommand{\dl}{\,\text{d}l}

\def\XXint#1#2#3{{\setbox0=\hbox{$#1{#2#3}{\int}$ }
\vcenter{\hbox{$#2#3$ }}\kern-.6\wd0}}

\newcommand{\beqn}{\begin{equation}}
\newcommand{\eeqn}{\end{equation}}

\newcommand{\bone}{\textbf{1}}
\newcommand{\bzero}{\textbf{0}}

\newcommand{\trans}{{\scriptscriptstyle\mskip-1mu\top\mskip-2mu}}

\newcommand{\II}{I\mskip-2mu I}

\newcommand{\sgn}{\operatorname{sgn}}

\title{Construction of an isometric immersion of a bounded, planar region from a framed curve}

\author{Brian Seguin$^{1}$ and Eliot Fried$^2$
\\[6pt]
\small $^{1}$Department of Mathematics and Statistics\\
\small Loyola University Chicago, Chicago, IL 60660-1537, USA\\
\\[6pt]
\small $^2$Mechanics and Materials Unit,\\
\small Okinawa Institute of Science and Technology\\
\small Onna, Okinawa, Japan 904-0495\\
}

\begin{document}
%\date{}

\maketitle

\begin{abstract}
\noindent
We develop a framework for characterizing isometric immersions of simply connected, bounded, planar regions with piecewise smooth boundaries into three-dimensional space. Each immersion is associated with a framed curve along the boundary of the image surface, comprised by a parametrized curve and a unit normal vector. We identify a set of compatibility and regularity conditions on this framed curve that ensure the existence of a $C^1$ isometric immersion that is $C^2$ almost everywhere and possesses finite bending energy. Under these conditions, we derive an exact dimensional reduction of the bending energy to a line integral over the boundary curve, without relying on asymptotic assumptions or approximations.  By analyzing the behavior of the unit normal vector along the framed boundary, we distinguish between planar and curved regions of the immersed surface. We identify the geometric conditions under which global $C^2$ regularity is potentially lost, in which case the associated immersion belongs to $W^{2,2}$---a Sobolev space that arises naturally in variational models of unstretchable elastic surfaces.\\

\noindent{\emph{Dedicated to Walter Noll on the occasion of his 100th birthday.}}

\end{abstract}

\tableofcontents

\vspace{.2in}

\section{Introduction}

We are concerned with the problem of determining the equilibrium configuration of an unstretchable elastic surface. The earliest contributions to this problem were made by Sadowsky~\cite{S301,S302}, who modeled the M\"obius band as a developable surface formed from a rectangular strip and proposed a variational approach to determine its equilibrium shape, deriving the corresponding Euler--Lagrange equations.\footnote{English translations of these papers were provided by Hinz and Fried~\cite{S302T,S301T}.} Because the material is unstretchable, the admissible deformations must preserve intrinsic distances and are therefore isometries. In current terminology, such deformations are said to be `isometric'.

Sadowsky considered the bending energy of a developable M\"obius band obtained from a rectangular strip through an isometric deformation and carried out a dimensional reduction under the assumption that the width of the band is small compared to its length. This led to an energy functional expressed as an integral over the midline of the band, for which he derived and analyzed the corresponding Euler--Lagrange equations. His derivation relied on the assumption that the curvature of the midline never vanishes. More recently, Freddi, Hornung, Mora, and Paroni~\cite{FHMP16} used the framework of $\Gamma$-convergence to justify a similar reduction while allowing the curvature of the midline to vanish. The limiting energy obtained in their work differs in general from the functional introduced by Sadowsky, but the two coincide if the curvature is everywhere positive and, at each point along the midline, exceeds the absolute value of the torsion. Geometrically, this condition ensures that twisting is sufficiently weak relative to bending, so that the midline remains close to planar at small scales.

Several decades after the publications of Sadowsky, Wunderlich~\cite{W62}\footnote{An English translation of this paper was provided by Todres~\cite{W62T}.} revisited the problem of determining the equilibrium configuration of a developable M\"obius band. Rather than pursuing an asymptotic analysis for small width-to-length ratios, he observed that every such band admits a ruled parameterization. Exploiting this structure, he expressed the bending energy as an exact integral along the midline, thereby obtaining a dimensionally reduced energy functional without imposing restrictions on the aspect ratio of the band. Wunderlich did not provide a minimization problem for determining an equilibrium configuration of a M\"obius band formed by subjecting a material surface in a rectangular reference configuration to an isometric deformation, nor did he derive the Euler--Lagrange equation for his dimensionally reduced bending energy. Seguin, Chen, and Fried~\cite{SCF20} formulated the minimization problem and proposed a structured method for its analysis and solution, which proceeds as follows. Let $\cD$ denote a flat, rectangular reference configuration of a material surface, and let $\bchi$ be an isometric deformation defined on the closure $\bar\cD$ of $\cD$. Granted that the material surface is isotropic and that its bending energy density depends quadratically on the curvature tensor, the total bending energy of the deformed surface $\bchi(\bar\cD)$ is given by
\beqn\label{origEintro}
E[\bchi]=2\omega\int_{\bchi(\bar\cD)}H^2\da,
\eeqn
where $\omega$ is the bending modulus of the material surface and $H$ is the mean curvature of $\bchi(\bar\cD)$. An equilibrium configuration corresponds to a minimizer of this energy over the admissible class
\beqn\label{ominprob}
\text{arg min}\{E[\bchi] \mid \bchi\in\cX\},
\eeqn
where $\cX$ denotes the set of all $C^2$ isometric deformations of $\bar\cD$ that satisfy the boundary conditions on its short edges required to produce a M\"obius band. The method proposed by Seguin, Chen, and Fried~\cite{SCF20} for solving this problem involves the following steps.

\begin{enumerate}

\item For each $\bchi \in \cX$, the deformed surface $\bchi(\bar\cD)$ is ruled and admits a parameterization of the form
\beqn\label{hrpar}
\hat\br(\alpha,\beta)=\bd(\alpha)+\beta\bg(\alpha),\quad (\alpha,\beta)\in\cP,
\eeqn
where $\bd$ is a parametrization of a curve on $\bchi(\bar\cD)$, referred to as the directrix, $\bg(\alpha)$ is a unit vector specifying the direction of the ruling that intersects $\bd$ at the point $\bd(\alpha)$, and $\cP$ is a parameter domain selected so that $\hat\br$ covers $\bchi(\bar\cD)$. There is a corresponding parameterization of the undeformed region $\bar\cD$ given by
\beqn
\hat\bx(\alpha,\beta)=\bc(\alpha)+\beta\bff(\alpha),\quad (\alpha,\beta)\in\cP,
\eeqn
where $\bc$ is a parametrization of a curve in $\bar\cD$, and $\bff$ is defined so that $\bg(\alpha)=\nabla\bchi(\bc(\alpha))\bff(\alpha)$ for all $\alpha$. The curve parametrized by $\bc$ can be chosen independently of the deformation $\bchi$.\footnote{In the present discussion, the focus is on M\"obius bands, but in the broader setting treated by Seguin, Chen, and Fried~\cite{SCF20}, the choice of $\bc$ depends on whether the deformed band is orientable or nonorientable.}

\item Substituting the representation \eqref{hrpar} into the energy expression \eqref{origEintro} yields a double integral over the parameter domain $\cP$, with integration performed first in $\beta$ and then in $\alpha$. The integration with respect to $\beta$ can be carried out explicitly, resulting in the reduced expression
\beqn\label{drbe}
E[\bchi]=\int_\cC\phi\dl,
\eeqn
where $\cC$ is the midline of the band, and the energy density $\phi$, which has dimensions of energy per unit length, depends on $\bd$ and $\bg$, their derivatives, and the width of $\bar\cD$. Since $\bar\cD$ is fixed, the bending energy can be regarded as a functional of the framed curve $(\bd,\bg)$; that is,
\beqn\label{hEdrbe}
\hat E[\bd,\bg]=\int_\cC\phi\dl.
\eeqn

\item The functional $\hat E$ is defined only on those framed curves that correspond to isometric deformations. Let $\cY$ denote the class of such framed curves. The associated variational problem takes the form
\beqn\label{rminprob}
\text{arg min}\{\hat E[\bd,\bg] \mid (\bd,\bg)\in\cY\}.
\eeqn
Seguin, Chen, and Fried~\cite{SCF20} derived the Euler--Lagrange equations corresponding to this problem. Although further analysis was not undertaken, the equations are amenable to analytical or numerical treatment.

\end{enumerate}
An advantage of the dimensionally reduced formulation of Seguin, Chen, and Fried~\cite{SCF20} is that the Euler--Lagrange equations associated with the functional $\hat E$ are ordinary differential equations, whereas the Euler--Lagrange equations for the original energy \eqref{ominprob} are partial differential equations. This simplification may facilitate analysis and also render the problem more tractable by direct methods in the calculus of variations. Given a solution $(\bd,\bg)\in\cY$ of the Euler--Lagrange equations for $\hat E$, Seguin, Chen, and Fried~\cite{SCF20} described a method for constructing a locally injective isometric deformation $\bchi$ of $\bar\cD$ such that the resulting surface $\bchi(\bar\cD)$ is represented by the parameterization \eqref{hrpar}.

The dimensionally reduced minimization problem \eqref{rminprob} is not equivalent to the original problem \eqref{ominprob}. This discrepancy arises for two reasons: (i) the isometric deformation constructed from a framed curve in $\cY$ is only locally injective, whereas elements of $\cX$ are globally injective; and (ii) the regularity of a deformation $\bchi\in\cX$ is $C^2$, while the isometries associated with framed curves in $\cY$ are only piecewise $C^2$. Nevertheless, every $\bchi\in\cX$ determines a corresponding framed curve in $\cY$, and the infimum of $\hat E$ over $\cY$ may be strictly less than the infimum of $E$ over $\cX$.

In this work, we extend the framework of Seguin, Chen, and Fried~\cite{SCF20} to more general reference configurations. Specifically, we take $\cD$ to be any simply connected planar region whose boundary $\partial\cD$ is piecewise smooth, with finitely many corners and finitely many subarcs along which the curvature vanishes. Allowing the curvature of $\partial\cD$ to vanish on finitely many subarcs reflects a desire to accommodate reference configurations with straight segments, in keeping with the broader variational framework of Freddi, Hornung, Mora, and Paroni~\cite{FHMP16}. We define $\cX$ to be the class of $C^2$ isometric immersions of $\bar\cD$ into $\mathbb{R}^3$; that is, locally injective isometries of class $C^2$.

A surface $\cS=\bchi(\cD)$ produced by an isometric immersion cannot generally be represented, in its entirety, by a ruled parameterization of the form \eqref{hrpar}. To ensure that each ruling intersects the directrix and that the choice of directrix remains independent of the immersion $\bchi$, we take $\bd$ to parameterize the boundary $\partial\cS$.\footnote{Chen, Fosdick, and Fried~\cite{CFF22} also used the entire boundary $\partial\cS$ of $\cS$ as a directrix.} Rather than using the ruling direction $\bg$, we frame $\partial \cS$ by pairing $\bd$ with the restriction of a unit normal $\bn$ to $\cS$. This choice has two advantages: $\bg$ is not uniquely defined along all of $\partial\cS$ and need not be continuous, even if $\cS$ is smooth. In contrast, the regularity of $\bn$ is tied directly to that of the immersion $\bchi$. Although it may seem natural to derive a dimensionally reduced bending energy using $\bd$ and $\bg$, we show that an analogous reduction can be achieved using $\bd$ and $\bn$.

The results presented here contribute to the understanding of isometric immersions defined on simply connected, planar regions. In this respect, they align with a broader body of work on such immersions and, more generally, on the geometry of developable surfaces. Foundational results concerning the structure of smooth developable surfaces were established by Hartman and Winter~\cite{HW51} and subsequently rederived by Massey~\cite{M61} using only elementary arguments. These results establish that any smooth developable surface can be decomposed into planar regions and curved regions, with the latter admitting a ruled parameterization. Pogorelov~\cite{P56} showed that this decomposition also applies under weaker regularity assumptions.

Since the image of a planar region under an isometric immersion is a developable surface, the structural results described above extend to isometric immersions. These results have been used to characterize the level sets of the gradient of such immersions. Building on the observation that isometric immersions with finite bending energy of the form \eqref{origEintro} belong to the Sobolev space $W^{2,2}$, several researchers --- including Kirchheim~\cite{K01}, Pakzad~\cite{MRP04}, M\"uller and Pakzad~\cite{MP05}, and Hornung~\cite{ELEH, PHFLSS} --- have examined the structure of immersions exhibiting lower regularity.

Our findings are also related Hornung's~\cite{H24} recent results concerning framed curves arising from local isometric immersions. His analysis begins with a curve embedded in the plane and demonstrates that any isometric immersion defined in a neighborhood of this curve gives rise to a framed curve satisfying a specified set of conditions. A partial converse is also established: any framed curve that satisfies these conditions corresponds to an isometric immersion defined in a neighborhood of the original curve. Whereas our work concerns $C^2$ immersions defined on simply connected planar regions, Hornung's work involves locally defined immersions of $W^{2,2}$ regularity.

The remainder of the paper is organized as follows. In Section~\ref{sectgeo}, we describe the geometric properties of an isometric immersion $\bchi$ defined on the closure $\bar\cD$ of a reference region $\cD$, with particular attention to the subset of $\bar\cD$ on which the second derivative $\nabla\nabla\bchi$ does not vanish. There, we introduce a class of parameterizations determined by the boundary curve $\partial\cS$ of the deformed region $\cS=\bchi(\cD)$, together with the restriction to $\partial\cS$ of a unit normal to $\cS$, thereby defining a framed curve. From the condition that $\bchi$ be an isometric immersion, we derive a set of compatibility and regularity requirements on the framed curve; these are formalized in Definition~\ref{admissfc}. In Section~\ref{sectrbe}, we express the bending energy of $\cS$ as a functional of the framed curve associated with $\partial\cS$. In  Section~\ref{constsur}, we prove our main result, namely that any framed curve satisfying the conditions of Definition~\ref{admissfc} determines a $C^1$ isometric immersion on $\bar\cD$ that is $C^2$ almost everywhere on $\cD$. In Section~\ref{sectcon}, we conclude with a summary and final remarks.

\section{Geometric properties of an isometric immersion}
\label{sectgeo}

Let $\cD\subseteq\cE$ be an open, bounded, simply connected, planar, two-dimensional surface embedded in a three-dimensional Euclidean point space $\cE$ with vector space $\cV$. We will refer to $\cD$ as the reference region, or reference for short. Assume that the boundary $\partial\cD$ of $\cD$ has the following properties:
\begin{itemize}
\item $\partial\cD$ is piecewise $C^2$, so that $\partial\cD$ has a finite number of corners.
\item The curvature of $\partial\cD$, if defined, vanishes on only a finite number of intervals.
\end{itemize}
Roughly speaking, $\partial\cD$ is any closed, simple curve that can be drawn by hand. Consider a $C^2$ isometric immersion $\bchi:\bar\cD\rightarrow\cE$ defined on the closure $\bar\cD$ of $\cD$ and set $\cS=\bchi(\cD)$.

Since $\bchi$ is an isometric immersion, a result due to Massey~\cite{M61} ensures that $\cD$ can be partitioned into subregions of two types:\footnote{There are also results along these lines in the case where $\bchi$ has less regularity. See Pogorelev \cite{P73} and M\"{u}ller and Pakzad \cite{MP05}.}

\begin{itemize}

\item Subsets of $\cD$ where $\nabla\nabla\bchi$ does not vanish. Each such subset is covered by non-intersecting straight line segments with continuously varying orientation and with endpoints on $\partial\cD$. Moreover, $\bchi$ maps each such line segment to a straight asymptotic curve of $\cS$.

\item Subsets of $\cD$ where $\nabla\nabla\bchi$ vanishes. Moreover, $\bchi$ rigidly maps such subsets to flat subsets of $\cS$.

\end{itemize}

We refer to the straight line segments of $\cD$ mentioned in the first bullet point as rulings, and we use the symbol $\cD_c$ to denote the subset of $\bar\cD$ comprised by the union of all such rulings, their endpoints included. The immersion $\bchi$ maps $\cD_c$ onto the curved subset  $\cS_c=\bchi(\cD_c)$ of $\bar\cS$. The straight asymptotic curves that cover $\cS_c$ are also rulings. To avoid ambiguity, we sometimes find it helpful to use the terms `reference ruling' and `image ruling' to distinguish the rulings of $\cD_c$ from those of $\cS_c$. To further clarify the notion of a reference ruling, we record the following observations:

\begin{itemize}

\item It is possible for two or more rulings to share an endpoint on $\partial\cD$. If this occurs, we say that the rulings involved meet at their shared endpoint.

\item The length of a ruling must be positive.

\item The straight line segments that bound $\cD_c$ are not considered to be rulings.

\end{itemize}
An example of such a partitioned reference region, including several reference rulings, is shown in Figure~\ref{Dfig}.

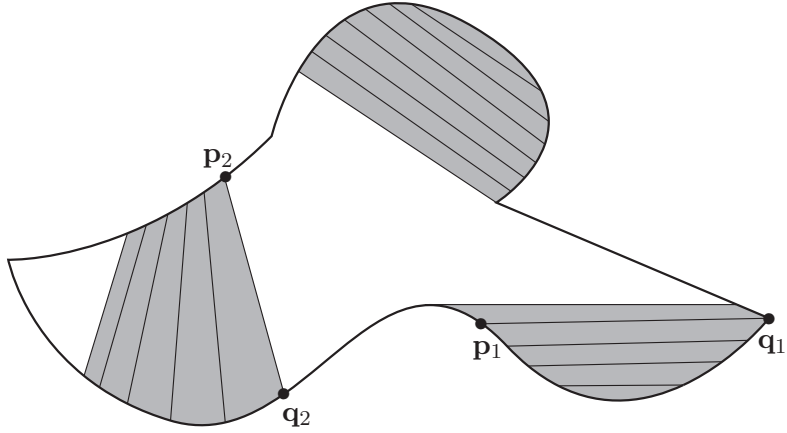
\begin{figure}[h]
\centering
\includegraphics[width=4in]{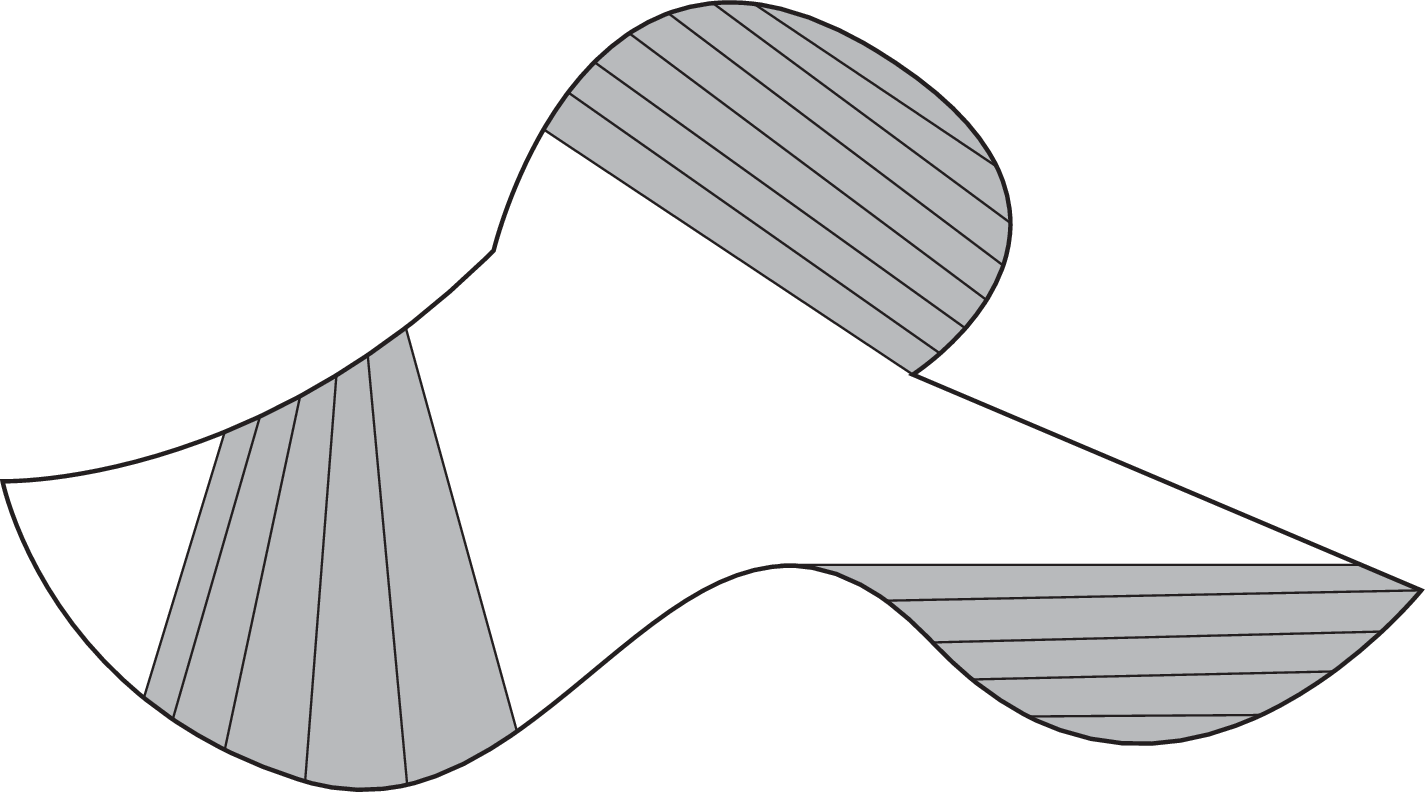}
\thicklines
\put(-209,91.5){$\bullet$}
\put(-215,100){$\bp_2$}
\put(-187,9.5){$\bullet$}
\put(-185,2){$\bq_2$}
\put(-112.5,36){$\bullet$}
\put(-113,27.5){$\bp_1$}
\put(-3.5,38){$\bullet$}
\put(-5,29.5){$\bq_1$}
\caption{Reference region $\cD$ partitioned into subregions according to whether the second derivative $\nabla\nabla\bchi$ of an isometric immersion $\bchi$ vanishes. The grey zones correspond to parts of $\cD$ where $\nabla\nabla\bchi$ is nonzero and are covered by interior reference rulings. The white zones correspond to flat parts of $\cD$ where $\nabla\nabla\bchi$ vanishes. Several reference rulings are drawn in the grey zones, including the line segment $[\bp_1,\bq_1]$. By contrast, the line segment $[\bp_2,\bq_2]$ lies on the boundary of a grey zone and is not a reference ruling.}
\label{Dfig}
\end{figure}

Although rulings can meet on $\partial\cD$, the next result shows that this occurs only under specific conditions that apply only infrequently.

\begin{proposition}\label{rulingprop}
Two rulings can meet on $\partial\cD$ only if they are collinear. Moreover, the set $\cI\subseteq\partial\cD$ of all points on $\partial\cD$ at which rulings meet is finite.
\end{proposition}

\begin{proof}
Contrary to the first part of the proposition, suppose that two rulings meet at a point $\bx\in\partial\cD$ but are not collinear. Then, those rulings must form an interior angle, with respect to $\bar\cD$, which differs from $\pi$. Suppose that the angle in question is strictly greater than $\pi$. Then, since the direction of rulings on $\cD_c$ vary continuously and since $\cD_c$ is open relative to $\bar\cD$, there is a neighborhood of $\bx$ within which rulings intersect, which is impossible, or meet at $\bx$, which is plausible only for angles that are strictly less than $\pi$. Granted that the angle at which the rulings meet at $\bx$ is strictly less than $\pi$, then, again since the direction of rulings of $\cD_c$ are vary continuously and since $\cD_c$ is open relative to $\bar\cD$, there exists a continuous array of rulings on one side of any ruling $\cR$ with an endpoint at $\bx$. Consider the corresponding collection of image rulings that meet at the image $\br=\bchi(\bx)$ on $\bar\cS$ of $\bx$, each of which is an asymptotic curve of $\cS$ and, thus, is a locus along which the unit normal to $\cS$ is constant. Since the unit normals along the image rulings that meet at $\br$ must agree, the portion of $\cS_c$ that is covered by those rulings must be planar. However, this contradicts the understanding that $\cS_c$ is curved. Two rulings that meet at a point $\bx\in\partial\cD$ must therefore be collinear, as proposed.

Contrary to the second part of the proposition, suppose that the set $\cI\subseteq\partial\bar\cD$ of all points at which rulings meet is infinite. A cluster point $\bx\in\partial\cD$ would then exist, surrounded by infinitely many points from $\cI$. The reference rulings associated with $\cI$ would then necessarily intersect within $\cD$. This, however, contradicts the understanding that rulings cannot intersect in $\cD$. The set $\cI$ must therefore be finite, as proposed.
\end{proof}

Since $\bchi$ is locally injective, for each $\bx\in\bar\cD$ there is a set $\cU$, open relative to $\bar\cD$, such that the restriction of $\bchi$  to $\cU$ is injective. Since $\bchi$ is $C^2$, the image $\bchi(\cU)$ of $\cU$ is a $C^2$ surface which can be oriented by identifying a $C^1$ unit-normal function. The additive inverse of the surface gradient of the unit-normal is the curvature tensor $\bL$ of $\bchi(\cU)$. Consider the pullback $\bH$, to $\cU$, of $\bL$, defined by
\beqn\label{Lpb}
\bH(\bx)=\nabla\bchi(\bx)^\trans\bL(\bchi(\bx))\nabla\bchi(\bx),\quad \bx\in\cU.
\eeqn
Notice that $\bH$ does not depend on the choice of $\cU$, as its definition via gradients relies on local properties. Thus, $\bH$ is defined throughout $\bar\cD$. Moreover, we can ensure that $\bH$ is continuous on $\bar\cD$ by appropriately selecting the orientations of the patches $\bchi(\cU)$ that constitute $\bar\cS$. Rather than partitioning $\cD$ using $\nabla\nabla\bchi$ as above, this partition is also determined by where $\bH$ does or does not vanish since $|\bH|=|\nabla\nabla\bchi|$.

\subsection{Parametrizations}\label{sectpar}

We next introduce parameterizations of almost all of $\cD_c$ and almost all of $\cS_c$. To achieve this, we begin by setting $\cM=[0,\ell]$, where $\ell$ designates the length of $\partial\cD$. We will use a periodic topology on $\cM$, whereby $0$ is identified with $\ell$ but the orientation associated with the interval is maintained. Given an arclength parameterization $\bc:\cM\rightarrow\partial\cD$ of $\partial\cD$, and let $\cF$ denote the finite subset of $\cM$ on which $\bc$ is not differentiable, it then follows that $\bd=\bchi\circ\bc$ is an arclength parameterization of $\partial\cS$ and that is $C^2$ except on $\cF$. 

Let $\{\bT,\bM,\bN\}$ be the Darboux frame of the curve parameterized by $\bc$ relative to $\bar\cD$ such that $\bT=\bc'$, $\bN$ is a unit normal to $\bar\cD$, and $\bM=\bN\times\bT$ points into $\bar\cD$. The structure equations for this frame are
\beqn\label{DFr}
\bT'=\kappa_\bc\bM,\qquad \bM'=-\kappa_\bc\bT,\qquad\text{and}\qquad \bN'=\bzero,
\eeqn
where $\kappa_\bc$ is the geodesic (or signed) curvature of $\bc$. These objects are only defined off of $\cF$, the points where $\bc$ is not differentiable, except for the normal $\bN$, which is defined on all of $\cM$. Although $\bN$ is defined on all of $\cM$, $\bT$ and $\bM$ are defined only on $\cM\setminus\cF$. Using $\bN$, we use $\theta[\bc,\bN](\alpha)$ to denote the external angle of the curve relative to $\bN$ at $\alpha\in\cM$. In other words, $\theta[\bc,\bN](\alpha)$ is the angle of rotation about $\bN$ that must be applied to the tangent vector $\lim_{\alpha'\uparrow\alpha}\bc'(\alpha')$ to obtain $\lim_{\alpha'\downarrow\alpha}\bc'(\alpha')$. Notice that on the set $\cM\setminus\cF$, the function $\theta[\bc,\bN]$ vanishes. For $\alpha\in\cF$, the previously mentioned limits exist since $\bc$ is piecewise $C^2$.

Letting $\cA$ be the subset of $\cM$ defined by
\begin{equation}\label{defA}
\cA=\{\alpha\in\cM\ |\ \text{there is a unique ruling in $\cD_c$ with endpoint $\bc(\alpha)$}\},
\end{equation}
and we define $\bff$ such that, for each $\alpha\in\cA$, $\bff(\alpha)$ is the unique unit vector that is directed into $\cD$ along the unique ruling that has $\bc(\alpha)$ as an endpoint: 
\beqn
|\bff(\alpha)|=1,
\qquad
0\le\bff(\alpha)\cdot\bM(\alpha)\le1,
\qquad
\bff(\alpha)\cdot\bN=0.
\eeqn
Finally, define $\bg$ such that, for each $\alpha\in\cA$,
\begin{equation}
\bg(\alpha)=(\nabla\bchi(\bc(\alpha)))\bff(\alpha)
\label{gdfn}
\end{equation}
points in the direction of the image ruling emanating from $\bd(\alpha)$. A result of Hartman and Wintner \cite{HW51} can be used to show that $\bg$ is $C^1$ on $\cA$. Since $\bchi$ is $C^2$ on $\bar\cD$, we see from \eqref{gdfn} that $\bff$ is also $C^1$ on $\cA$.

For each $\alpha\in\cA$, let $\hat\beta:\cA\to(0,\infty)$ be defined such that $\hat\beta(\alpha)$ is the length of the reference ruling with endpoint $\bc(\alpha)$. Then, since $\bchi$ is an isometric immersion, the image ruling with endpoint $\bd(\alpha)$ also has length $\hat\beta(\alpha)$. From the geometry of $\partial\cD$, we infer that, for each $\alpha\in\cA$,
\beqn\label{bbetarep}
\hat\beta(\alpha)=\min\{ |\bc(\alpha')-\bc(\alpha)|\ | \ \alpha'\in\cM\ \text{and}\ \bc(\alpha')=\bc(\alpha)+\beta\bff(\alpha)\ \text{for}\ \beta>0\}.
\eeqn
For any subset $\cA'\subseteq\cA$, let $\cP(\cA')$ be given by
\beqn\label{parc}
\cP(\cA')=\{(\alpha,\beta)\in\cM\times\Real\ |\ \alpha\in\cA',\ \beta\in[0,\hat\beta(\alpha)]\},
\eeqn
and consider the parameterization $\hat\bx$ defined such that
\beqn\label{hx}
\hat\bx(\alpha,\beta)=\bc(\alpha)+\beta\bff(\alpha),\qquad (\alpha,\beta)\in\cP(\cA).
\eeqn
From \eqref{parc}, we see that the image of $\hat\bx$ is a subset of $\cD_c$. We also notice that $\hat\bx$ is not a bijection since it is not one-to-one on $\cP(\cA)$.

Since the image $\cS_c$ of $\cD_c$ is also ruled, the portion of $\cS_c$ that is determined by the image of $\hat\bx$ admits a parametrization $\hat\br$ of the form
\beqn\label{hr}
\hat\br(\alpha,\beta)=\bd(\alpha)+\beta\bg(\alpha),\qquad (\alpha,\beta)\in\cP(\cA).
\eeqn
The parameterizations $\hat\bx$ and $\hat\br$ must satisfy $\hat\br(\alpha,\beta)=\bchi(\hat\bx(\alpha,\beta))$ for $(\alpha,\beta)\in\cP(\cA)$ and, thus, do not provide a complete characterization the immersion $\bchi$. It should be emphasized that $\hat\bx$ and $\hat\br$ do not parameterize the entire sets $\cD_c$ and $\cS_c$, respectively.

By construction, and wherever $\bc$ is differentiable,
\beqn\label{parnorm}
|\bc'|=|\bd'|=1\quad\text{and}\quad |\bff|=|\bg|=1.
\eeqn
Moreover, since $\bchi$ is an isometric immersion, the metric coefficients of $\hat\bx$ and $\hat\br$ must agree, from which we conclude that $\bc$, $\bff$, $\bd$, and $\bg$ must satisfy 
\beqn\label{parcond}
\quad |\bff'|=|\bg'|,\qquad \bc'\cdot\bff=\bd'\cdot\bg,\qquad\text{and}\qquad \bc'\cdot\bff'=\bd'\cdot\bg'.
\eeqn

Let $\{\bt,\bm,\bn\}$ be the Darboux frame of the curve parameterized by $\bd$ relative to $\bar\cS$, so that $\bt=\bd'$, $\bn$ is a unit normal to $\bar\cS$, $\bm=\bn\times\bt$, and $\bm$ points into $\bar\cS$. The structure equations of this frame are 
\beqn\label{DFs}
\bt'=\kappa_g\bm+\kappa_n\bn,\qquad \bm'=-\kappa_g\bt+\tau_g\bn,\qquad\text{and}\qquad \bn'=-\kappa_n\bt-\tau_g\bm,
\eeqn
where $\kappa_g$, $\kappa_n$, and $\tau_g$ are the geodesic curvature, normal curvature, and geodesic torsion, respectively, of the curve parameterized by $\bd$ relative to $\bar\cS$. Although $\bn$ is defined on all of $\cM$, $\bt$ and $\bm$ are defined only on $\cM\setminus\cF$.

Since $\bchi$ is an isometry, $\nabla\bchi$ must satisfy
\beqn\label{orthdefn}
(\nabla\bchi)^{\trans}\nabla\bchi=\bone,
\eeqn
where $\bone$ denotes the two-dimensional identity tensor. We see from \eqref{Lpb} and \eqref{orthdefn} that the quantities
\beqn
H=\tfrac{1}{2}\tr\, \bH\qquad\text{and}\qquad K=\det\bH
\label{HandK}
\eeqn
determine the mean and Gaussian curvatures of $\cS$ as functions of points in $\bar\cD$. The restrictions of $H$ and $K$ to $\partial\cD$ are related to $\kappa_g$, $\kappa_n$, and $\tau_g$ by the identity
\beqn\label{conSid}
2\kappa_n(\alpha)H(\bc(\alpha))-K(\bc(\alpha))=\kappa_n^2(\alpha)+\tau_g^2(\alpha),
\eeqn
a proof of which is provided by, for instance, Beetle~\cite{Beetle}. Of course, since $\bar\cD$ is planar and $\bchi$ is an isometry, $K=0$ on $\bar\cD$ and \eqref{conSid} reduces to
\beqn\label{conSid2}
2\kappa_n(\alpha)H(\bc(\alpha))
=\kappa_n^2(\alpha)+\tau_g^2(\alpha).
\eeqn
From \eqref{DFs}$_3$ and \eqref{conSid}, we find that
\beqn\label{nccond9}
\kappa_n=0
\quad\Longrightarrow\quad
\bn'={\rm \bf 0}
\quad\text{on}\quad
\cM\setminus\cF.
\eeqn

Being an isometric immersion, $\bchi$ must preserve the geodesic curvature of any $C^1$ curve on $\bar\cD$. Hence, in particular, the geodesic curvature $\kappa_g$ of the curve parametrized by $\bd$ relative to $\bar\cS$ must be equal to the geodesic curvature $\kappa_\bc$ of $\partial\cD$; thus, by \eqref{DFs}$_1$,
\beqn\label{currel2}
\kappa_g=(\bd'\times\bd'')\cdot\bn=\kappa_\bc.
\eeqn
Moreover, $\bchi$ must preserve angles between any pair of linearly independent vectors in the tangent space at each point on $\bar\cD$. In particular, bearing in mind that $\bn$ is defined for all $\alpha\in\cM$, the value $\theta[\bd,\bn](\alpha)$ at $\alpha\in\cM$ of the exterior angle of the curve parametrized by $\bd$ relative to $\bar\cS$ must satisfy
\beqn\label{extangle}
\theta[\bd,\bn](\alpha)=\theta[\bc,\bN](\alpha).
\eeqn

Since $\bg$ is a unit vector that points into and is tangent to $\cS$, there exist functions $g_\bt$ and $g_\bm$ defined on $\cA$ such that
\beqn\label{defg}
g_\bt^2+g_\bm^2=1,\qquad g_\bm\ge0, \qquad \text{and}\qquad \bg=g_\bt\bt+g_\bm\bm.
\eeqn
Since $\bg$, $\bt$, and $\bm$ are $C^1$ on $\cA\setminus\cF$, so are $g_\bt$ and $g_\bm$. Differentiating \eqref{defg}$_1$, we see that
\beqn\label{gcompconst}
g_\bt g_\bt'+g_\bm g_\bm'=0.
\eeqn
Furthermore, differentiating \eqref{defg}$_3$ and using \eqref{DFs}, we find that $\bg'$ admits a representation of the form
\beqn\label{g'rep0}
\bg'=(g_\bt'-g_\bm\kappa_g)\bt+(g_\bm'+g_\bt\kappa_g)\bm+(g_\bt\kappa_n+g_\bm\tau_g)\bn.
\eeqn
However, since the ruled parameterization \eqref{hr} describes part of $\bar\cS$, which  is a developable surface, $\bg'$ must satisfy\footnote{See, for example, Struik \cite[\S5-5]{S61}.}
\beqn\label{gK=0}
\bg'\cdot\bn=0.
\eeqn
Applying \eqref{gK=0} to \eqref{g'rep0}, we obtain the identity
\beqn\label{K=0}
g_\bt\kappa_n+g_\bm\tau_g=0
\eeqn
and, hence, arrive at a reduced version of \eqref{g'rep0}:
\beqn\label{g'rep}
\bg'=(g_\bt'-g_\bm\kappa_g)\bt+(g_\bm'+g_\bt\kappa_g)\bm.
\eeqn
Noticing from \eqref{g'rep} that
\beqn\label{g'repc}
|\bg'|^2=(\kappa_g-g_\bm g_\bt'+g_\bt g_\bm')^2,
\eeqn
we introduce the notation
\beqn\label{defG}
G=\kappa_g-g_\bm g_\bt'+g_\bt g_\bm'
\eeqn
which, with the aid of \eqref{gcompconst}, allows us to express $\bg'$ as
\beqn\label{g'rep2}
\bg'=-G(g_\bm\bt-g_\bt\bm).
\eeqn
Since image rulings are asymptotic curves, which have zero normal curvature, and a ruling is parallel to $\bd$ exactly if $g_\bm=0$, Meusnier's theorem\footnote{See, for example, Struik \cite[\S2-5]{S61}.} implies that
\beqn\label{asymcond}
\kappa_n=0
\quad\Longleftrightarrow\quad g_\bm=0 
\quad\text{on}\quad\cA.
\eeqn

\begin{remark}
{\rm With reference to \eqref{defg}, it would be permissible to introduce an angle $\theta\in[0,\pi]$, depending on $\alpha$, such that $g_\bt=\cos\theta$ and $g_\bm=\sin\theta$. If this were done, then $G$ would be given by $G=\kappa_g+\theta'$. This approach has been taken by others. See, for example, Chen, Fosdick, and Fried~\cite{CFF22}.}
\end{remark}

\subsection{Endpoints of rulings}

For each $\alpha\in\cA$, consider the ruling with endpoint $\bc(\alpha)$ and orientation $\bff(\alpha)$. The other endpoint of this ruling is also on $\partial\cD$. It is thus possible to implicitly define a function\footnote{This function was first conceived of and used by Chen, Fosdick, and Fried~\cite{CFF22}.} $\mu:\cA\rightarrow\cM$ such that, given $\alpha\in\cA$,
\beqn\label{mudef}
\bc(\mu(\alpha))=\bc(\alpha)+\hat\beta(\alpha)\bff(\alpha).
\eeqn
Since reference rulings are transformed into image rulings, we must also have
\beqn\label{M10}
\bd(\mu(\alpha))=\bd(\alpha)+\hat\beta(\alpha)\bg(\alpha),\quad \alpha\in\cA.
\eeqn
We next establish some useful properties of $\hat\beta$ and $\mu$.

\begin{proposition}\label{bbetamureg1}
For each $\alpha\in\cA\setminus\cF$ such that $\mu(\alpha)\in\cA\setminus\cF$ and $g_\bm(\mu(\alpha))>0$,
\begin{itemize}
\item $\mu$ is an involution in the sense that
\beqn
\mu(\mu(\alpha))=\alpha,
\label{mureflexive}
\eeqn
\item there exists a neighborhood of $\alpha$ upon which $\hat\beta$ and $\mu$ are $C^1$,
\item $\mu'$ satisfies the inequality
\beqn
\mu'(\alpha)\le0,
\label{muprimele0}
\eeqn
which applies strictly if $g_\bm(\alpha)>0$.
\end{itemize}
\end{proposition}

\begin{proof}
Fix $\alpha_0\in\cA\setminus\cF$ and assume that $\mu(\alpha_0)\in\cA\setminus\cF$ and that $g_\bm(\mu(\alpha_0))>0$. Since $\mu(\alpha_0)\in\cA$, there is a unique ruling that emanates from $\bc(\mu(\alpha_0))$ and must be identical to the ruling that emanates from $\bc(\alpha_0)$. Thus, we confirm that \eqref{mureflexive} holds for $\alpha=\alpha_0$.

Next, consider the function $\bh$ defined by
\beqn
\bh(\alpha,\beta,\gamma)=\bc(\alpha)+\beta\bff(\alpha)-\bc(\gamma),\qquad (\alpha,\beta,\gamma)\in\cM\times \Real\times \cM.
\label{hdfn}
\eeqn
Referring to the definition \eqref{mudef} of $\mu$, we infer that, for any choice of $\alpha\in\cA$ that is sufficiently close to $\alpha_0$,
\beqn\label{gcond}
\bh(\alpha,\hat\beta(\alpha),\mu(\alpha))=\bzero.
\eeqn
Bearing in mind that there is a neighborhood of $\alpha_0$ in which $\bc$ and $\bff$ are $C^1$, we see from \eqref{hdfn} that $\bh$ is $C^1$ in a neighborhood of $(\alpha_0,\hat\beta(\alpha_0),\mu(\alpha_0))$. Let $\hat\bh(\beta,\gamma)=\bh(\alpha_0,\beta,\gamma)$, and notice that the gradient $\nabla\hat\bh$ of $\hat\bh$ is a linear mapping from $\Real^2$ to a two-dimensional vector space. The linear transformation arising from evaluating $\nabla \hat\bh$ at $(\hat\beta(\alpha_0),\mu(\alpha_0))$ is given by
\beqn\label{gradh}
\nabla \hat\bh (\hat\beta(\alpha_0),\mu(\alpha_0)) (u_1,u_2)=u_1\bff(\alpha_0)-u_2\bc'(\mu(\alpha_0)),\qquad (u_1,u_2)\in\Real^2.
\eeqn
Since $g_\bm(\mu(\alpha_0))>0$ by hypothesis, the ruling emanating from $\bc(\alpha_0)$ cannot be tangent to $\partial\cD$ at $\bc(\mu(\alpha))$. It follows that $\bff(\alpha_0)$ is not parallel to $\bc'(\mu(\alpha))$ and, thus, that $\bc'(\mu(\alpha_0))$ and $\bff(\alpha_0)$ are linearly independent. Hence, from \eqref{gradh}, the linear transformation $\nabla\hat\bh(\hat\beta(\alpha_0),\mu(\alpha_0))$ is invertible. Invoking the implicit function theorem,\footnote{See, for example, Spivak \cite[Chapter 2]{Spivak65}.} we thus infer that there are unique $C^1$ functions $\tilde\beta$ and $\tilde\mu$ such that
\beqn\label{gcond2}
\bh(\alpha,\tilde\beta(\alpha),\tilde\mu(\alpha))=\bzero.
\eeqn
for all $\alpha\in\cA$ in some neighborhood of $\alpha_0$. From the uniqueness part of the implicit function theorem and \eqref{gcond}, $\tilde\beta=\hat\beta$ and $\tilde\mu=\mu$. Thus, $\hat\beta$ and $\mu$ are $C^1$ in a neighborhood of $\alpha_0$.

To establish the remainder of the result, consider the definition of the derivative $\mu'$ at $\alpha_0$:
\beqn
\mu'(\alpha_0)=\lim_{\alpha\rightarrow\alpha_0}\frac{\mu(\alpha)-\mu(\alpha_0)}{\alpha-\alpha_0}.
\label{muprime0}
\eeqn
If $\alpha$ approaches $\alpha_0$ from the left, then, to ensure that the ruling connecting $\bc(\alpha)$ to $\bc(\mu(\alpha))$ does not cross the ruling connecting $\bc(\alpha_0)$ to $\bc(\mu(\alpha_0))$, $\mu(\alpha)$ must necessarily approach $\mu(\alpha_0)$ from the right. Thus, the the right-hand side of \eqref{muprime0} is nonpositive. Similarly, if $\alpha$ approaches $\alpha_0$ from the right, then $\mu(\alpha)$ approaches $\mu(\alpha_0)$ from the left, so the right-hand side of \eqref{muprime0} is again nonpositive. Hence, $\mu'(\alpha)\le0$. Finally, if  $g_\bm(\alpha_0)>0$, then, since $\alpha_0=\mu(\mu(\alpha_0))$, we see that $g_\bm(\mu(\mu(\alpha_0)))>0$ and, hence, that $\mu$ is differentiable at $\mu(\alpha_0)$. Differentiating
\beqn
\mu(\mu(\alpha))=\alpha\quad\text{for $\alpha\in\cA\setminus\cF$ close to $\alpha_0$}
\eeqn
with respect to $\alpha$ and evaluating at $\alpha=\alpha_0$, we find that $\mu'(\alpha_0)\ne0$ and, thus, that the strict version of the inequality \eqref{muprimele0} holds.
\end{proof}

Although the parameterization $\hat\bx$ is not one-to-one, we next show that, if restricted to a certain domain, $\hat\bx$ is two-to-one.

\begin{proposition}\label{prop2t1}
The parameterization $\hat\bx$ is two-to-one on $\cP(\cA\cap\mu^{-1}(\cA))$.
\end{proposition}

\begin{proof}
If $\alpha\in\cA\cap\mu^{-1}(\cA)$, then the ruling emanating from $\bc(\alpha)$ and terminating at $\bc(\mu(\alpha))$ is identical to the ruling emanating from $\bc(\mu(\alpha))$ and terminating at $\bc(\alpha)$. Thus, 
\beqn
\hat\bc(\alpha,\beta)=\hat\bc(\mu(\alpha),\hat\beta(\alpha)-\beta),\qquad \beta\in [0,\hat\beta(\alpha)].
\eeqn
Suppose that, for some parameter pair $(\alpha',\beta')\in\cP(\cA\cap\mu^{-1}(\cA))$ distinct from $(\alpha,\beta)$, $\hat\bc(\alpha',\beta')=\hat\bc(\alpha,\beta)$.
Then, the ruling emanating from $\bc(\alpha)$ must intersect the ruling emanating from $\bc(\alpha')$, which is impossible since rulings cannot intersect in $\cD$. Since each ruling of $\cD_c$ is unique, we conclude with reference to \eqref{hx} that the claim holds.
\end{proof}

\subsection{Representation of the mean curvature}

Our next objective is to derive a formula for the mean curvature $H$ associated with the parameterization $\hat\br$, which is defined by \eqref{HandK}$_1$. The resulting expression only holds at points of $\cD_c$ where $\hat\br$ is regular. In conjunction with the results at the end of the previous subsection, this motivates us to consider a parameter set smaller than $\cA$. However, even on this smaller set, the parameterizations $\hat\bx$ and $\hat\br$ cover almost all of the curved parts of $\bar\cD$ and $\bar\cS$, respectively. 

To obtain the desired formula for $H$, begin by noticing that a normal vector $\bnu(\alpha,\beta)$ at $\hat\br(\alpha,\beta)$ is given by
\beqn\label{bnu}
\bnu(\alpha,\beta)=\bn(\alpha),\qquad (\alpha,\beta)\in\cP(\cA).
\eeqn
We may thus use the parameterization $\hat\br$ to express $H$ as
\beqn
H=\frac{\tr(I^{-1}\II)}{2},
\eeqn
where $I$ and $\II$ are the first and second fundamental forms of the portion of $\cS_c$ that is covered by $\hat\br$. Using \eqref{hr} in the conventional formulae
\beqn
I=\left(
\begin{array}{cc}
|\hat\br_\alpha|^2 & \hat\br_\alpha\cdot\hat\br_\beta \\[4pt]
\hat\br_\alpha\cdot\hat\br_\beta & |\hat\br_\beta|^2
\end{array}
\right ),\qquad
\II=\left (
\begin{array}{cc}
-\bnu_\alpha\cdot\hat\br_\alpha & -\bnu_\alpha\cdot\hat\br_\beta \\[4pt]
-\bnu_\alpha\cdot\hat\br_\beta & -\bnu_\beta\cdot\hat\br_\beta
\end{array}
\right),
\eeqn
where subscripts are used to denote partial differentiation with respect to the parameters $\alpha$ and $\beta$. Invoking \eqref{gK=0}, and taking into consideration that, by \eqref{bnu},
\beqn
\bnu_\alpha=\bn'\qquad\text{and}\qquad \bnu_\beta=\bzero,
\eeqn
we find that
\beqn\label{Hrep0}
H(\hat\bx(\alpha,\beta))=-\frac{\bn'(\alpha)\cdot(\bt(\alpha)+\beta\bg'(\alpha))}{2[|\bt(\alpha)+\beta\bg'(\alpha)|^2-(\bt(\alpha)\cdot\bg(\alpha))^2]}.
\eeqn
Furthermore, utilizing \eqref{defg}, \eqref{gcompconst}, and \eqref{g'rep2}, we find that \eqref{Hrep0} reduces to
\beqn\label{Hrep2}
H(\hat\bx(\alpha,\beta))=\frac{\kappa_n(\alpha)+\beta G(\alpha)(\tau_g(\alpha) g_\bt(\alpha)-\kappa_n(\alpha) g_\bm(\alpha))}{2(g_\bm(\alpha)-\beta G(\alpha))^2}.
\eeqn
The representation \eqref{Hrep2} for $H$ holds as long as $\alpha\not\in\cF$ and the parameterization $\hat\br$ is regular at $(\alpha,\beta)$, meaning that
\beqn\label{jacobian}
|\hat\br_\alpha(\alpha,\beta)\times\hat\br_\beta(\alpha,\beta)|=|g_\bm(\alpha)-\beta G(\alpha)|\ne0.
\eeqn
From \eqref{hx}, \eqref{parcond} and \eqref{g'rep2}, we find that
\beqn\label{jacobian2}
|\hat\bx_\alpha(\alpha,\beta)\times\hat\bx_\beta(\alpha,\beta)|=|g_\bm(\alpha)-\beta G(\alpha)|,
\eeqn
from which we infer that $\hat\bx$ is regular exactly where $\hat\br$ is regular. To determine conditions under which the requirement \eqref{jacobian} is met, we first establish a preliminary result, which stems from the stipulation that rulings cannot intersect in $\cD$.

\begin{proposition}\label{propinjcondp}
The inequality
\beqn\label{injcondp}
g_\bm(\alpha)-\hat\beta(\alpha)G(\alpha)\ge0
\eeqn
holds at each $\alpha\in\cA\setminus\cF$.
\end{proposition}

\begin{proof}
Given $\alpha\in\cA\setminus\cF$, consider the alternatives $\bff'(\alpha)=\bzero$ and $\bff'(\alpha)\ne\bzero$. If $\bff'(\alpha)=\bzero$, then $G(\alpha)=0$ by \eqref{parcond}$_1$ and \eqref{g'rep2}. Thus, \eqref{injcondp} reduces to $g_\bm\ge0$, which holds by \eqref{defg}$_2$. If, instead, $\bff'(\alpha)\ne\bzero$, there exists an $\alpha'\in\cA$ belonging to a neighborhood of $\alpha$ such that $\bff(\alpha')$ is not parallel to $\bff(\alpha)$. The infinite straight lines that contain the rulings that emanate from $\bc(\alpha)$ and $\bc(\alpha')$ must therefore intersect, and the identity
\beqn
\bc(\alpha)+\beta\bff(\alpha)
=\bc(\alpha')+\beta'\bff(\alpha')
\label{cor6a}
\eeqn
must be satisfied for some choice of $\beta$ and $\beta'$. Since $\bff(\alpha')\cdot\bff'(\alpha')=0$, resolving \eqref{cor6a} in the direction of $\bff'(\alpha')$ gives
\beqn\label{prop3}
\bff'(\alpha')\cdot(\bc(\alpha')-\bc(\alpha))=\beta\bff'(\alpha')\cdot(\bff(\alpha)-\bff(\alpha')).
\eeqn
Since $\bff'(\alpha')\ne\bzero$, for $\alpha'$ in a sufficiently small neighborhood of $\alpha$, $\bff'(\alpha')\cdot(\bff(\alpha)-\bff(\alpha'))\ne0$. Hence, \eqref{prop3} is equivalent to
\beqn
\frac{\bff'(\alpha')\cdot(\bc(\alpha')-\bc(\alpha))}{\bff'(\alpha')\cdot(\bff(\alpha)-\bff(\alpha'))}=\beta.
\label{cor6b}
\eeqn
Since rulings cannot intersect in $\cD$, $\beta$ must satisfy either $\beta>\hat\beta(\alpha)$ or $\beta<0$. Thus, by \eqref{cor6b}, either
\beqn
\frac{\bff'(\alpha')\cdot(\bc(\alpha')-\bc(\alpha))}{\bff'(\alpha')\cdot(\bff(\alpha)-\bff(\alpha'))}>\hat\beta(\alpha)
\qquad\text{or}\qquad \frac{\bff'(\alpha')\cdot(\bc(\alpha')-\bc(\alpha))}{\bff'(\alpha')\cdot(\bff(\alpha)-\bff(\alpha'))}<0
\label{cor6c}
\eeqn
must hold. On passing to the limit $\alpha'\to\alpha$, \eqref{cor6c} becomes
\beqn\label{prop3.1}
-\frac{\bff'(\alpha)\cdot\bc'(\alpha)}{|\bff'(\alpha)|^2}\ge\hat\beta(\alpha)
\qquad\text{or}\qquad 
-\frac{\bff'(\alpha)\cdot\bc'(\alpha)}{|\bff'(\alpha)|^2}\le0,
\eeqn
which, by \eqref{parcond}$_{1,3}$, \eqref{defg}$_3$, and \eqref{g'repc}, is equivalent to
\beqn\label{injectivitycond}
\frac{g_\bm(\alpha)}{G(\alpha)}\ge\hat\beta(\alpha)\qquad\text{or}\qquad \frac{g_\bm(\alpha)}{G(\alpha)}\le0.
\eeqn
To explore the implications of the alternatives in \eqref{injectivitycond}, recall from \eqref{defg}$_2$ that $g_\bm(\alpha)\ge0$. If \eqref{injectivitycond}$_1$ applies, then $G(\alpha)>0$ and \eqref{injcondp} holds. If, alternatively, \eqref{injectivitycond}$_2$ applies, then $G(\alpha)<0$ and \eqref{injcondp} holds.
\end{proof}

We can now state a condition that ensures that the representation \eqref{Hrep2} for the mean curvature $H$ is valid.

\begin{corollary}\label{injcondcor}
For any fixed $\alpha\in\cA\setminus\cF$, if $g_\bm(\alpha)>0$ and $g_\bm(\mu(\alpha))>0$, then
\beqn\label{injcondc}
g_\bm(\alpha)-\beta G(\alpha)\ne0,\quad \beta\in[0,\hat\beta(\alpha)].
\eeqn
\end{corollary}

\begin{proof}
Differentiating \eqref{M10} with respect to $\alpha$, taking the vector product of the resulting condition with $\bg(\alpha)$, and using \eqref{defg} and \eqref{g'rep2} yields the identity
\begin{align}\label{muTcond1}
\mu'(\alpha)\bd'(\mu(\alpha))\times\bg(\alpha)
&=(g_\bm(\alpha)-\hat\beta(\alpha)G(\alpha))\bn(\alpha).
\end{align}
Since $g_\bm(\mu(\alpha))>0$, $\bd'(\mu(\alpha))$ and $\bg(\alpha)$ are not parallel and $\mu'(\alpha)\ne0$ by Proposition~\ref{bbetamureg1}. It follows that the left-hand side of \eqref{muTcond1} cannot vanish and, thus, that $g_\bm(\alpha)-\hat\beta(\alpha)G(\alpha)\ne0$. To obtain \eqref{injcondc} for $\beta\in[0,\hat\beta(\alpha))$, it suffices to consider two alternatives depending on the sign of $G(\alpha)$. If $G(\alpha)\le0$, then \eqref{injcondc} holds since $g_\bm(\alpha)> 0$ and $\beta\ge0$. If $G(\alpha)>0$, then \eqref{injcondp} can be used to show that for $\beta<\hat\beta(\alpha)$, it follows that
\beqn
g_\bm(\alpha)-\beta G(\alpha)>g_\bm(\alpha)-\hat\beta(\alpha) G(\alpha)\ge0,
\eeqn
whereby \eqref{injcondc} holds.
\end{proof}

Although the expression \eqref{Hrep2} for the mean curvature is not valid for all $\alpha\in\cA$, the next result ensures that the set of $\alpha$ for which \eqref{Hrep2} does not hold consists of a finite number of elements.

\begin{proposition}\label{hAprop}
The set
\beqn
\{\alpha\in\cA\setminus\cF\ |\ g_\bm(\alpha)=0\}
\eeqn
has finite cardinality.
\end{proposition}

\begin{proof}
Choose $\alpha\in\cA\setminus\cF$ so that $g_\bm(\alpha)=0$. Then, by the lower bound \eqref{defg}$_2$, $g_\bm'(\alpha)=0$. Thus, by  \eqref{defG} and Proposition~\ref{propinjcondp},
\beqn
G(\alpha)=\kappa_g(\alpha)=\kappa_\bc(\alpha)\le0.
\eeqn
Moreover, \eqref{Hrep2} specializes to
\beqn
H(\hat\bx(\alpha,\beta))
=\frac{\kappa_n(\alpha)+\beta\kappa_\bc(\alpha)\tau_g(\alpha)}{2\beta^2\kappa^2_\bc(\alpha)}
\label{prop7a}
\eeqn
and is valid for each choice of $\beta$. Suppose that $\kappa_\bc(\alpha)<0$. Then, by \eqref{conSid2} and \eqref{asymcond}, \eqref{prop7a} reduces to $H(\hat\bx(\alpha,\beta))=0$, which is not possible since the image of $\hat\bx$ is contained in $\cD_c$, on which $H$ cannot vanish. Next, suppose that $\kappa_\bc(\alpha)=0$. Recall that the signed curvature $\kappa_\bc$ of $\partial\cD$ can vanish on only a finite number of intervals. Then, since $\kappa_\bc=\kappa_g$, it follows that $\kappa_g$ can only vanish on a finite number of intervals. Considering the definition \eqref{defA} of $\cA$, and recalling that $g_\bm(\alpha)=0$, $\alpha$ cannot be interior to any of these intervals. Since there are only a finite number of intervals where $\kappa_g$ vanishes, it follows that $g_\bm$ can vanish at only a finite number of points.
\end{proof}

Propositions \ref{prop2t1} and \ref{injcondcor} motivate us to consider the subset of $\cA\setminus\cF$ defined according to
\beqn
\hat\cA=\{\alpha\in\cA\setminus\cF\mskip4mu|\mskip4mu \mu(\alpha)\not\in\cF,\ g_\bm(\alpha)\ne0, g_\bm(\mu(\alpha))\ne0\}.
\label{Ahat}
\eeqn
Since $\cF$ has finite cardinality, we see from Proposition~\ref{hAprop} that $\hat\cA$ and $\cA$ differ by a set with a finite number of elements. Since $g_\bm>0$ on $\hat\cA$, we can use \eqref{K=0} to infer that 
\beqn\label{K=0con}
\tau_g=-\frac{\kappa_ng_\bt}{g_\bm}.
\eeqn
Using \eqref{K=0con}, the representation \eqref{Hrep2} of $H$ can be expressed as
\beqn\label{Hrep4}
H(\hat\bx(\alpha,\beta))=\frac{\kappa_n(\alpha)}{2g_\bm(\alpha)(g_\bm(\alpha)-\beta G(\alpha))}.
\eeqn
It should be emphasized that \eqref{Hrep4} is only valid on $\hat\cA$ rather than on all of $\cA$.

We conclude this subsection by showing that the restriction of $\hat\bx$ to $\cP(\hat\cA)$ covers almost all of $\cD_c$.

\begin{proposition}\label{measurezero}
The set $\hat{\cD}_c=\hat\bx(\cP(\hat\cA))\subseteq\cD_c$ differs from $\cD_c$ by a set of zero areal measure.
\end{proposition}

\begin{proof}
Points in $\cD_c\setminus\hat{\cD}_c$ lie on rulings that are tangent to $\partial\cD$ or have an endpoint at which the unit tangent $\bt$ to $\partial\cD$ is discontinuous. Since $\bar\cD$ has at most a finite number of corners, it follows from Propositions~\ref{rulingprop} and \ref{hAprop} that the area measure of $\cD_c\setminus\hat\cD_c$ is zero.
\end{proof}

As a corollary of Proposition~\ref{measurezero}, it can be shown that $\hat\br(\cP(\hat\cA))$ differs from $\cS_c$ by a set of zero areal measure.

\subsection{Centrality of the framed boundary of the image surface} \label{Edn}

We have seen that the curved part of $\cS$ is characterized, up to a set of zero areal measure, by the parameterizations $\hat\bx$ and $\hat\br$ on $\cP(\hat\cA)$. These parameterizations are determined by four functions of the arclength $\alpha$: $\bc$, $\bff$, $\bd$, and $\bg$. The arclength parametrization $\bc$ of $\partial\cD$ is determined solely by the shape of $\cD$. Consequently, it is independent of the isometric immersion $\bchi$ and, thus, the other functions. However, these remaining three functions are not independent, as they must fulfill the isometry constraints \eqref{parcond}, along with other conditions. We next focus on expressing $\bff$ and $\bg$ in terms of $\bd$ and the unit normal $\bn$ to $\cS$ along $\partial\cS$.

Associated with the boundary $\partial\cD$ of $\cD$ is the Darboux frame $\{\bT,\bM,\bN\}$. Together with the frame $\{\bt,\bm,\bn\}$ determined by $\bd$ and $\bn$, we find from \eqref{parcond}$_2$ and \eqref{defg} that
\beqn\label{deffp0}
\bff=(\bg\cdot\bd')\bT+\sqrt{1-(\bg\cdot\bd')^2}\mskip2mu\bM=g_\bt\bT+g_\bm\bM.
\eeqn
It is thus evident that $\bff$ and $\bg$ are not independent; instead, they are determined reciprocally.
 
Focusing on the set $\hat\cA$ defined in \eqref{Ahat}, we see from \eqref{K=0con} that
\beqn\label{n'rep}
\bn'=-\kappa_n\Big(\bt- \frac{g_\bt}{g_\bm}\bm\Big)
\eeqn
and, thus, invoking \eqref{defg}, that
\beqn\label{grepm}
\bn'\times\bn=\frac{\kappa_n}{g_\bm}\bg.
\eeqn
From \eqref{grepm}, we find that
\beqn\label{defgp0}
\bg=\sgn(\kappa_n)\frac{\bn'\times\bn}{|\bn'\times\bn|}.
\eeqn
Although $\bn'$ may not be $C^1$, it is evident from \eqref{defgp0} that the ratio $\bn'\times\bn/|\bn'\times\bn|$ is $C^1$. Moreover, since $\bn'\times\bn/|\bn'\times\bn|$ is $C^1$ and $\bn'$ is orthogonal to $\bn$, $\bn'/|\bn'|$ is also $C^1$.

After expressing $\bg$, and subsequently $\bff$, in terms of $\bd$ and $\bn$, \eqref{bbetarep} can be used to determine $\hat\beta$ and \eqref{mudef} can be used to determine $\mu$. Additionally, as we next demonstrate, the set $\hat\cA$ defined in \eqref{Ahat} can be characterized completely by $\bd$ and $\bn$.

\begin{proposition}
The set $\hat\cA$ upon which the representation \eqref{Hrep4} of the mean curvature $H$ holds is given by
\beqn\label{althA}
\hat\cA=\{\alpha\in\cM\setminus\cF\mskip4mu|\mskip4mu \mu(\alpha)\not\in\cF,\ \kappa_n(\alpha)\ne0, \kappa_n(\mu(\alpha))\ne0\}.
\eeqn
\end{proposition}

\begin{proof}
Denote the set on the right-hand side of \eqref{althA} by $\cB$. Suppose that $\alpha\in\hat\cA$, so that $g_\bm(\alpha)\ne0$ and $g_\bm(\mu(\alpha))\ne0$. It then follows from \eqref{asymcond} that $\alpha\in\cB$. Suppose, alternatively, that $\alpha\in\cB$. Since the Gaussian curvature of $\bar\cS$ vanishes, \eqref{conSid2} holds, and, since $\kappa_n(\alpha)\ne0$, it follows that $H(\bc(\alpha))\ne0$. Hence, there is a reference ruling that has $\bc(\alpha)$ as an endpoint. Any such ruling cannot be tangent to $\partial\cD$ at $\bc(\alpha)$ since $\kappa_n(\alpha)\ne0$. Since rulings can only meet on $\partial\cD$ if they are on the same line and, hence, tangent to $\partial\cD$, there is a unique ruling emanating from $\bc(\alpha)$ and it, thus, follows that $\alpha\in\cA$ and $g_\bm(\alpha)\ne0$. This ensures that $\mu(\alpha)$ is well-defined. Moreover, the same reasoning applies at $\mu(\alpha)$, so we must also have $g_\bm(\mu(\alpha))\ne0$. Hence, $\alpha\in\hat\cA$.
\end{proof}

From the findings presented in this subsection, it appears that the curved part of $\cS$, given a reference region $\cD$, is entirely determined by the curve parametrized by $\bd$ and the unit normal $\bn$ of the associated Darboux frame. In Section~\ref{constsur}, we will demonstrate that such a framed curve determines the isometric immersion $\bchi$ on the entirety of $\bar\cD$, rather than on $\hat\cD_c$ alone. However, a framed curve need not correspond to an isometric immersion of $\bar\cD$; to do so it must satisfy several conditions, one being that an associated measure of bending energy must be finite. To articulate this prerequisite, it is essential to derive an expression for the energy stored in bending $\bar\cD$ to $\bar\cS$ directly in terms of $\bd$, $\bn$, and related quantities. 

%%%%%%%%%%%%%%%%%%%%%

\section{Reduction of the bending energy to a line integral}\label{sectrbe}

The (dimensionless) bending energy associated with the isometric immersion $\bchi$ is given by
\beqn\label{BendEnergy}
E=2\int_{\bar\cD}H^2\da=2\int_{\cD_c}H^2\da,
\eeqn
where we have relied on the understanding that $H=0$ on $\bar\cD\setminus\cD_c$. Notice that $E$ is finite since $\bchi$ is $C^2$ and, hence, $H$ is bounded on $\bar\cD$. We next show that $E$ can be reduced to an integral over $\cM$. 

Utilizing Propositions~\ref{prop2t1} and \ref{measurezero}, the expression \eqref{jacobian2} for the Jacobian of the parametrization $\hat\bx$, and the representation \eqref{Hrep4} of the mean curvature $H$, we find that  \eqref{BendEnergy} can be reduced to an integral over $\hat\cA$. Toward this objective, we first observe that applying the area formula\footnote{See, for example, Ambrosio, Fusco, and Pallara \cite[\S2.10]{AFP00}.} and keeping in mind that $\hat\bx$ is two-to-one yields
\begin{align}
E&=2\int_{\hat\bx(\cP(\hat\cA))} H^2\da
\notag\\[4pt]
&=\iint_{\cP(\hat\cA)}H^2(\hat\bx(\alpha,\beta))|\hat\bx_\alpha(\alpha,\beta)\times\hat\bx_\beta(\alpha,\beta)|\dbeta\dalpha
\notag\\[4pt]
&=\frac{1}{4}\int_{\hat\cA}\Big(\int_0^{\hat\beta(\alpha)}\frac{\kappa^2_n(\alpha)\dbeta}{g^2_\bm(\alpha)(g_\bm(\alpha)-\beta G(\alpha))}\Big)\dalpha.
\label{derGnz0}
\end{align}
Given $\alpha\in\hat\cA$, the value of the inner integral in \eqref{derGnz0} depends on whether $G(\alpha)\ne0$ or $G(\alpha)=0$:
\beqn\label{intGcases}
\int_0^{\hat\beta(\alpha)}\frac{\kappa^2_n(\alpha)\dbeta}{g^2_\bm(\alpha)(g_\bm(\alpha)-\beta G(\alpha))}=
\begin{cases}
\displaystyle
\frac{\kappa^2_n(\alpha)}{G(\alpha)g^2_\bm(\alpha)}\ln\frac{g_\bm(\alpha)}{ g_\bm(\alpha)-\hat\beta(\alpha) G(\alpha)}, & G(\alpha)\ne0, 
\\[16pt]
\displaystyle
\frac{\hat\beta(\alpha)\kappa^2_n(\alpha)}{g^3_\bm(\alpha)}, & G(\alpha)=0.
\end{cases}
\eeqn
Since
\beqn\label{Gez}
\lim_{G\rightarrow 0}\frac{\kappa_n^2}{Gg_\bm^2}\ln\frac{g_\bm}{ g_\bm-\hat\beta G}=\frac{\hat\beta\kappa_n^2}{g_\bm^3},
\eeqn
we see that the integral \eqref{intGcases} depends continuously on $G$. With this understanding, we conclude that \eqref{BendEnergy} admits a dimensionally reduced representation of the form
\beqn\label{derGnz}
E=\frac{1}{4}\int_{\hat\cA} \frac{\kappa_n(\alpha)^2}{G(\alpha)g_\bm(\alpha)^2}\ln\frac{g_\bm(\alpha)}{g_\bm(\alpha)-\hat\beta(\alpha) G(\alpha)}\dalpha.
\eeqn
In writing \eqref{derGnz}, we emphasize that the limiting expression \eqref{Gez} should be used when $G(\alpha)=0$.

The derivation of this dimensionally reduced bending energy rests on findings presented in Section~\ref{sectgeo}, which lead to the conclusion that there are a finite number of values $\alpha$ for which there is a ruling tangent to $\partial\cD$ at $\bc(\alpha)$ or $\bc(\mu(\alpha))$. Moreover, the values of $\alpha$ for which $\bn'(\alpha)=\bzero$ should not be integrated over, as those values correspond to the flat portions of $\cS$. Thus, $E$ can be expressed as
\beqn\label{RedBEmod}
E=\int_\cM \phi \dl,
\eeqn
where $\phi$ is defined such that
\beqn\label{RedED}
\phi(\alpha)=
\begin{cases}
\displaystyle
\frac{\omega\kappa_n(\alpha)^2}{4G(\alpha)g_\bm(\alpha)^2}\ln\frac{g_\bm(\alpha)}{g_\bm(\alpha)-\hat\beta(\alpha) G(\alpha)}, & \alpha\in\hat\cA,
\\[14pt]
0, & \bn'(\alpha)=\bzero,
\\[10pt]
\infty, & \alpha\not\in\hat\cA \ \text{and}\ \bn'(\alpha)\ne\bzero.
\end{cases}
\eeqn
Consider an $\alpha\in\cM$ such that the third alternative for $\phi$ is realized. Recalling \eqref{althA}, we see that  $\bn'(\alpha)\ne\bzero$ and at least one of the following conditions holds: $\kappa_n(\alpha)=0$, $\kappa_n(\mu(\alpha))=0$, $\alpha\in\cF$, or $\mu(\alpha)\in\cF$. It follows from the fact that $\cF$ is has finite cardinality and from Proposition~\ref{hAprop} that this is possible only for a finite number of $\alpha$. Thus, setting $\phi$ to be infinite for such $\alpha$ does not affect the value of the dimensionally reduced expression \eqref{RedBEmod} for $E$. 

Although \eqref{BendEnergy} depends on the function $\bchi$ as defined in \eqref{RedED}, from the discussion in Subsection~\ref{Edn}, \eqref{RedED} is completely determined by $\bd$ and $\bn$.

%%%%%%%%%%%%%%%%%%%%%%%%%%%%%%%%

\section{Construction of an isometric immersion from a framed curve}
\label{constsur}

Hereafter, we focus on demonstrating that, under specific conditions, a framed curve can be used to construct an isometric immersion from $\bar\cD$. Not any framed curve can be used for such a construction. The results in Section~\ref{sectgeo} and Section~\ref{sectrbe} motivate what conditions the framed curve must satisfy to carry out the construction.

For the rest of this section, we assume that the reference region $\cD$ is fixed and that it satisfies the assumptions stipulated at the beginning of Section~\ref{sectgeo}. Recall that $\cM=[0,\ell]$ is endowed with a periodic topology and that $\bc:\cM\rightarrow\partial\cD$ is an arclength parameterization of the boundary $\partial\cD$ of $\cD$. Recall, also, that $\bc$ is piecewise $C^2$ and that the finite number of points at which $\bc$ is not differentiable is denoted by $\cF$. In what follows, there is no given isometric immersion under consideration.

Utilizing the notation from \eqref{DFr}, we next provide a definition of a framed curve with singular set $\cF$, while ensuring the necessary regularity.

\begin{definition}\label{fcss}
A framed curve, with singular set $\cF$, consists of a pair of functions $(\bd,\bn)$ defined such that $\bd:\cM\rightarrow\cE$ parameterizes a closed curve that is piecewise $C^2$ with $\cF$ being the set of points where $\bc'$ and $\bc''$ fail to exist and $\bn:\cM\rightarrow\cV$ is a continuous unit-vector function that is $C^1$ off of $\cF$ satisfying $\bd'\cdot\bn=0$ on $\cM\setminus\cF$, with $\bn'/|\bn'|$ being $C^1$ wherever it is defined.
\end{definition}

Given a framed curve $(\bd,\bn)$, we introduce $\bt=\bd'$ and $\bm=\bn\times\bt$. Then, $\{\bt,\bm,\bn\}$ is an orthonormal frame associated with the curve parameterized by $\bd$ and the notational conventions established in \eqref{DFs} can be employed without confusion. Introducing the set
\beqn
\tilde\cA=\{\alpha\in\cM\setminus\cF\ |\ \kappa_n(\alpha)\ne0\},
\label{cAdef}
\eeqn
we define $\bg:\tilde\cA\to\cV$ and $\bff:\tilde\cA\to\cV$ through
\beqn
\bg=\sgn(\kappa_n)\frac{\bn'\times\bn}{|\bn'|}
\label{defgp}
\eeqn
and
\beqn
\bff=(\bg\cdot\bd') \bT+\sqrt{1-(\bg\cdot\bd')^2}\mskip2mu\bM.
\label{deffp}
\eeqn
Moreover, define $\hat\beta:\tilde\cA\rightarrow\Real$ by
\beqn\label{bbetarep4}
\hat\beta(\alpha)=\min\{ |\bc(\alpha')-\bc(\alpha)|\ | \ \alpha'\in\cM\ \text{and}\ \bc(\alpha')=\bc(\alpha)+\beta\bff(\alpha)\ \text{for}\ \beta>0\}.
\eeqn
and $\mu:\tilde\cA\rightarrow\cA$ implicitly by
\beqn\label{mudef4}
\bc(\mu(\alpha))=\bc(\alpha)+\hat\beta(\alpha)\bff(\alpha).
\eeqn
Since $\bg\cdot\bd'\ne1$ on $\tilde\cA$, we see from \eqref{deffp} that $\bff$ can never be tangent to $\partial\cD$ and points into $\cD$. This ensures that $\hat\beta$ and $\mu$ are well-defined.

Since $\bg$ is orthogonal to $\bn$, it has a representation, relative to $\{\bt,\bm,\bn\}$, of the form
\beqn\label{defgg}
\bg=g_\bt\bt+g_\bm\bm.
\eeqn
Moreover, since $|\bg|=1$ and since $\kappa_n\ne0$ on $\tilde\cA$, we infer that $g_\bt$ and $g_\bm$ must satisfy
\beqn\label{defgg2}
g_\bt^2+g_\bm^2=1
\qquad\text{and}\qquad 
g_\bm>0.
\eeqn
Differentiating $\bg$ and using the notational conventions introduced in \eqref{DFs} and \eqref{defG} we find that $\bg':\tilde\cA\to\cV$ admits a representation of the form
\beqn
\bg'=-G(g_\bm\bt-g_\bt\bm)+(g_\bg\kappa_n+g_\bm\tau_g)\bn,
\label{g'repp0}
\eeqn
Recognizing from \eqref{defgp} that $\bg'\cdot\bn=0$, we find that \eqref{g'repp0} reduces to
\beqn\label{g'repp}
\bg'=-G(g_\bm\bt-g_\bt\bm).
\eeqn

We are now prepared to introduce the subclass of framed curves that are suitable for generating isometric immersions.

\begin{definition}\label{admissfc}
A framed curve $(\bd,\bn)$ is \emph{admissible} if:
\begin{itemize}
\item \emph{(compatibility with $\partial\cD$)} $\bd$ and $\bn$ satisfy\footnote{See the notation introduced in \eqref{DFr} and \eqref{DFs}.}
\begin{align}
\kappa_g=\kappa_\bc&\quad \text{on}\quad \cM\setminus\cF 
\qquad\text{and}\qquad \theta[\bc,\bN]=\theta[\bd,\bn]
\quad \text{on}\ \cF,\label{cdncond} 
\\[4pt]
&\kappa_n=0\quad\text{implies that}\quad\bn'={\rm\bf0}\quad\text{on}\quad \cM\setminus\cF;
\label{cdncond2}
\end{align}

\item \emph{(properties of $\mu$)} the following conditions involving $\mu$ hold:
\begin{enumerate}[label=\textbf{\textit{M\arabic*}}]

\item \label{M1} $\bd(\mu(\alpha))=\bd(\alpha)+\hat\beta(\alpha)\bg(\alpha)$ for all $\alpha\in\tilde\cA$,

\item \label{M2} For all $\alpha_1\in\tilde\cA$ and $\alpha_2\in\tilde\cA$, the line segments connecting $\bc(\alpha_1)$ to $\bc(\mu(\alpha_1))$ and $\bc(\alpha_2)$ to $\bc(\mu(\alpha_2))$ coincide, share an endpoint, or are disjoint.

\end{enumerate}

\item \emph{(finite energy)} the reduced bending energy \eqref{RedBEmod}--\eqref{RedED}, with $\hat\cA=\tilde\cA\cap\mu^{-1}(\tilde\cA)$, is finite.

\end{itemize}

\end{definition}
\noindent We will refer to line segments mentioned in \emph{\ref{M2}} as rulings, and we will use the notation $[\bc(\alpha),\bc(\mu(\alpha))]$ to denote the ruling connecting $\bc(\alpha)$ to $\bc(\mu(\alpha))$. Considering \eqref{RedED}, we recognize that a tacit feature of the assumption that the energy \eqref{RedBEmod}--\eqref{RedED} being finite is that $g_\bm$, $\hat\beta$, and $G$ satisfy
\beqn\label{aecond}
g_\bm(\alpha)-\hat\beta(\alpha)G(\alpha)>0\quad \text{for almost every}\ \alpha\in\hat\cA.
\eeqn

Before we delineate the steps needed to construct a surface from an admissible framed curve $(\bd,\bn)$, we establish a few preliminary results. The first such result yields a useful alternative to expression \eqref{deffp} for $\bff$ and an equally useful expression for $\bff'$.

\begin{proposition}\label{idcond}
Given an admissible framed curve $(\bd,\bn)$, with $\bg$ and $\bff$ defined in \eqref{defgp} and \eqref{deffp}, respectively, it follows that
\beqn\label{frepp}
\bff=g_\bt\bT+g_\bm\bM
\eeqn
and that
\beqn\label{f'repp}
\bff'=-G(g_\bm\bT-g_\bt\bM).
\eeqn
Moreover, \eqref{parnorm} and \eqref{parcond} hold on $\tilde\cA$.
\end{proposition}

\begin{proof}
The representation \eqref{frepp} of $\bff$ follows from the consequences $\bff\cdot\bT=\bg\cdot\bd'=g_\bt$ and $\bff\cdot\bM=\sqrt{1-(\bg\cdot\bd')^2}=\sqrt{1-g_\bt^2}=g_\bm$ of \eqref{deffp}, \eqref{defgg}, and \eqref{defgg2}. To obtain \eqref{f'repp}, begin by differentiating \eqref{frepp} and use \eqref{cdncond}$_1$ to find that
\beqn\label{f'repp0}
\bff'=(g_\bt'-g_\bm\kappa_g)\bT+(g_\bm'+g_\bt\kappa_g)\bM.
\eeqn
Next, recalling the definition of $G$ in \eqref{defG} and using \eqref{defgg2}$_1$ yields
\beqn\label{Gcalc}
Gg_\bm=g_\bm\kappa_g-g_\bt'\quad\text{and}\quad Gg_\bt=g_\bm'+g_\bt\kappa_g.
\eeqn
Putting together \eqref{f'repp0} and \eqref{Gcalc} results in \eqref{f'repp}.

The conditions \eqref{parnorm} and \eqref{parcond}$_{2}$ follow from the assumed properties of $\bc$ and $\bd$ and the definitions of $\bg$ and $\bff$. Also, the remaining conditions \eqref{parcond}$_{1,3}$ follow from \eqref{g'repp} and \eqref{f'repp}.
\end{proof}

We next show that on $\hat\cA$ rulings cannot share an endpoint. Notice that the set
\beqn\label{hcAdef}
\hat\cA=\{\alpha\in\cM\setminus\cF\ |\ \mu(\alpha)\not\in\cF,\ \kappa_n(\alpha)\ne0, \text{ and } \kappa_n(\mu(\alpha))\ne0 \}
\eeqn
is analogous to the set with the same name that was previously introduced in \eqref{althA}. Moreover, $\alpha\in\hat\cA$ if and only if the ruling $[\bc(\alpha),\bc(\mu(\alpha))]$ is not tangent to $\partial\cD$ at either of its endpoints. Since $\bg$ is continuous on $\tilde\cA$ and $\partial\cD$ is a continuous curve, it follows that $\hat\cA$ is an open set.

\begin{proposition}\label{M2strong}
Given an admissible framed curve $(\bd,\bn)$ and $\alpha_1\in\hat\cA$ and $\alpha_2\in\hat\cA$, the rulings $[\bc(\alpha_1),\bc(\mu(\alpha_1))]$ and $[\bc(\alpha_2), \bc(\mu(\alpha_2))]$ cannot share an endpoint.
\end{proposition}

\begin{proof}
Contrary to the proposition, assume that $[\bc(\alpha_1),\bc(\mu(\alpha_1))]$ and $[\bc(\alpha_2),\bc(\mu(\alpha_2))]$ share an endpoint. By the definition of $\hat\cA$, the endpoints of these rulings cannot be corners of the boundary $\partial\cD$ of $\cD$. There are four mutually exclusive ways in which the rulings can meet:
\begin{enumerate}
\item $\bc(\alpha_1)=\bc(\alpha_2)$,
\item $\bc(\mu(\alpha_1))=\bc(\alpha_2)$,
\item $\bc(\mu(\alpha_2))=\bc(\alpha_1)$,
\item $\bc(\mu(\alpha_1))=\bc(\mu(\alpha_2))$.
\end{enumerate}
Since $\bc$ is injective, the first possibility implies that $[\bc(\alpha_1),\bc(\mu(\alpha_1))]$ and $[\bc(\alpha_2),\bc(\mu(\alpha_2))]$ coincide. 

To address the remaining possibilities, we separately consider the situations in which $[\bc(\alpha_1),\bc(\mu(\alpha_1))]$ and $[\bc(\alpha_2),\bc(\mu(\alpha_2))]$ are or are not parallel. 

If $[\bc(\alpha_1),\bc(\mu(\alpha_1))]$ and $[\bc(\alpha_2),\bc(\mu(\alpha_2))]$ are parallel, then they must be tangent to $\partial\cD$ at their shared endpoint. However, this cannot occur since $\alpha_1\in\hat\cA$ and $\alpha_2\in\hat\cA$.

Suppose, alternatively, that $[\bc(\alpha_1),\bc(\mu(\alpha_1))]$ and $[\bc(\alpha_2),\bc(\mu(\alpha_2))]$ are not parallel and that the second meeting option holds. Since the direction of rulings is continuous and \emph{\ref{M2}} ensures that these rulings can only meet at an endpoint, there is an interval of parameter values on one side of $\alpha_1$ such that all of the rulings associated with these parameter values must meet at $\bc(\mu(\alpha_1))$. It follows from this and that since $\hat\cA$ is open, that this interval can be made small enough so that it is contained in $\hat\cA$. 

If $\alpha$ and $\alpha'$ are any parameters in this interval, then
\beqn
\bc(\alpha)+\hat\beta(\alpha)\bff(\alpha)=\bc(\alpha')+\hat\beta(\alpha')\bff(\alpha').
\eeqn
Taking the dot product of this equation with $\bff'(\alpha')$ and solving for $\hat\beta(\alpha)$ results in
\beqn
\hat\beta(\alpha)=\frac{(\bc(\alpha')-\bc(\alpha))\cdot\bff'(\alpha')}{(\bff(\alpha)-\bff(\alpha'))\cdot\bff'(\alpha')}.
\eeqn
Passing to the limit $\alpha'\to\alpha$ yields
\beqn
\hat\beta(\alpha)=-\frac{\bc'(\alpha)\cdot\bff'(\alpha)}{|\bff'(\alpha)|^2}.
\eeqn
Next, use \eqref{f'repp} to find that
\beqn
0=g_m(\alpha)-\hat\beta(\alpha)G(\alpha).
\eeqn
Since this must hold on a subinterval of $\hat\cA$, this violates \eqref{aecond}, so this case is not possible.

The third and fourth meeting options are handled in a similar way. Thus, the claim holds.
\end{proof}

We next establish useful properties of the functions $\mu$, $\hat\beta$, $\bff$, and $\bg$ when they are restricted to $\hat\cA$.

\begin{proposition}\label{muprop}
The functions $\mu$, $\hat\beta$, $\bff$, and $\bg$ satisfy
\beqn
\mu(\mu(\alpha))=\alpha,
\quad
\hat\beta(\mu(\alpha))=\hat\beta(\alpha),
\quad
\bff(\mu(\alpha))=-\bff(\alpha),
\quad\text{and}\quad
\bg(\mu(\alpha))=-\bg(\alpha),
\label{prop14}
\eeqn
respectively, for each $\alpha\in\hat\cA$.
\end{proposition}

\begin{proof}
Suppose, contrary to the proposition, that $\mu(\mu(\alpha))\ne\alpha$ for some $\alpha\in\hhA$. Then, $[\bc(\alpha),\bc(\mu(\alpha))]$ and $[\bc(\mu(\alpha)),\bc(\mu(\mu(\alpha)))]$ must meet at $\bc(\mu(\alpha))$, which would violate Proposition~\ref{M2strong}. Thus, \eqref{prop14}$_1$ must hold for each $\alpha\in\hhA$, per the proposition. Next, by the definition \eqref{mudef4} of $\mu$,
\begin{align}
\bc(\mu(\mu(\alpha)))
&=\bc(\mu(\alpha))
+\hat\beta(\mu(\alpha))\bff(\mu(\alpha))
\notag\\[4pt]
&=\bc(\alpha)+\hat\beta(\alpha)\bff(\alpha)
+\hat\beta(\mu(\alpha))\bff(\mu(\alpha)),
\qquad
\alpha\in\hat\cA.
\end{align}
Thus, since, by \eqref{prop14}$_1$, $\bc(\mu(\mu(\alpha)))=\bc(\alpha)$ for each $\alpha\in\cA$, $\hat\beta$ and $\bff$ must satisfy
\beqn
\hat\beta(\mu(\alpha))\bff(\mu(\alpha))+\hat\beta(\alpha)\bff(\alpha)=\bzero,
\qquad
\alpha\in\hhA.
\label{421}
\eeqn
However, since $|\bff|=1$ by \eqref{deffp} and since $\hat\beta>0$ by definition, \eqref{421} holds only if \eqref{prop14}$_{2,3}$ hold for each $\alpha\in\hat\cA$, per the proposition. An analogous argument relying on \ref{M1} leads to the conclusion that \eqref{prop14}$_4$ holds for each $\alpha\in\hat\cA$, per the proposition.
\end{proof}

The next result establishes the regularity of both $\hat\beta$ and $\mu$. Its proof is essentially the same as that of Proposition~\ref{bbetamureg1} and, thus, is omitted.

\begin{proposition}\label{bbetamureg}
The functions $\hat\beta$ and $\mu$ are $C^1$ on $\hhA$. Moreover, $\mu'<0$ on $\hhA$.
\end{proposition}

\noindent Our final preliminary result provides information beyond that in \eqref{aecond}.

\begin{proposition}\label{locinv}
Given $\alpha\in\tilde\cA$, the functions $g_\bm$ and $G$ must satisfy the inequality
\beqn\label{injcondcc0}
g_\bm(\alpha)-\beta G(\alpha)>0.
\eeqn
for each $\beta\in[0,\hat\beta(\alpha))$.
Moreover, $g_\bm$ and $G$ must satisfy the inequality
\beqn\label{injcondcc}
g_\bm(\alpha)-\hat\beta(\alpha) G(\alpha)>0
\eeqn
for each $\alpha\in\hat\cA$.

\end{proposition}

\begin{proof}
In view of \emph{\ref{M2}}, \eqref{aecond}, and the continuity of the functions $g_\bm$, $\hat\beta$, and $G$, an argument reminiscent of that used to prove Proposition~\ref{propinjcondp} can be used to establish the inequality
\beqn\label{injcondp.2}
g_\bm-\hat\beta G\ge 0\quad \text{on}\ \tilde\cA.
\eeqn

To establish \eqref{injcondcc0}, let $\alpha\in\tilde\cA$ and separately consider the cases $G(\alpha)\le0$ and $G(\alpha)>0$. If $G(\alpha)\le0$, then \eqref{injcondcc} clearly holds since $g_\bm(\alpha)> 0$ and $\beta\ge0$. If $G(\alpha)>0$, \eqref{injcondp.2} can be used to infer that, for $\alpha\in\tilde\cA$ and $\beta<\hat\beta(\alpha)$,
\beqn
g_\bm(\alpha)-\beta G(\alpha)>g_\bm(\alpha)-\hat\beta(\alpha) G(\alpha)\ge 0.
\eeqn
Thus, \eqref{injcondcc0} holds. 

To establish \eqref{injcondcc}, let $\alpha\in\hat\cA$, compute the vector product of the derivative of \emph{\ref{M1}} with $\bg(\alpha)$, and use \eqref{defgg}, \eqref{defgg2}, and \eqref{g'repp} to find that
\begin{align}\label{muTcond}
\mu'(\alpha)\bd'(\mu(\alpha))\times\bg(\alpha)
&=(g_\bm(\alpha)-\hat\beta(\alpha)G(\alpha))\bn(\alpha).
\end{align}
Since $g_\bm(\mu(\alpha))\ne0$ for $\mu(\alpha)\in\tilde\cA$, $\bd'(\mu(\alpha))$ cannot be parallel to $\bg(\mu(\alpha))$. Moreover, since $\bg(\mu(\alpha))=-\bg(\alpha)$ by Proposition~\ref{muprop}, $\bd'(\mu(\alpha))$ also cannot be parallel to $\bg(\alpha)$. In combination with the understanding, from Proposition~\ref{bbetamureg}, that $\mu'(\alpha)\ne0$ for $\alpha\in\hat\cA$, it follows that the left-hand side of \eqref{muTcond} cannot vanish. Thus, $g_\bm(\alpha)-\hat\beta(\alpha)G(\alpha)\ne0$ and it follows that \eqref{injcondcc} holds for $\alpha\in\hat\cA$.
\end{proof}

We are now positioned to establish the main result of this section.

\begin{theorem}\label{fSconstP}
Given an admissible framed curve $(\bd,\bn)$, there is a $C^1$ isometric immersion $\bchi:\bar\cD\rightarrow\cE$ such that $\bd=\bchi\circ\bc$ and $\bn(\alpha)$ is normal to $\bchi(\bar\cD)$ at $\bd(\alpha)$ for each $\alpha\in\cM$. Moreover, $\chi$ is $C^2$ almost everywhere.
\end{theorem}

\begin{proof}
The proof is divided into six steps, beginning with showing that a pair of parameterizations suffice to define $\bchi$, as an isometry, on a portion of $\bar\cD$. This step is followed by five additional steps in which the initial findings are progressively broadened to: extend the definitions of the functions $\bff$, $\bg$, and $\mu$ to include the boundary of the set $\hat\cA$ defined in \eqref{hcAdef}; show that the original definition of $\bchi$ can be extended, maintaining its status as an isometry, including what remains of $\bar\cD$; establish the $C^1$ regularity of $\bchi$ on $\cD$; demonstrate that the $C^1$ regularity of $\bchi$ extends to the closure $\bar\cD$ of $\cD$; and, finally, show that $\bchi$ has $C^2$ regularity almost everywhere.

%\emph{Step 1}: 

\begin{proofpart}

The aim in this step of the proof is to show that two parameterizations suffice to describe an isometric immersion on a portion of $\bar\cD$. Consider parameterizations $\hat\bx$ and $\hat\br$ defined such that
\beqn\label{hxthm}
\hat\bx(\alpha,\beta)=\bc(\alpha)+\beta\bff(\alpha),\qquad (\alpha,\beta)\in\cP(\hat\cA)
\eeqn
and
\beqn\label{hrthm}
\hat\br(\alpha,\beta)=\bd(\alpha)+\beta\bg(\alpha),\qquad (\alpha,\beta)\in\cP(\hat\cA).
\eeqn
It is now shown that $\hat\bx$ is a two-to-one function on $\cP(\hhA)$. Let $\hat\cD_c$ denote the set $\hat\bx(\cP(\hhA))$, and consider $\bx_0\in\hat\cD_c$. By the definition of $\hat\cD_c$, $\hat\bx(\alpha_0,\beta_0)=\bx_0$ for some $(\alpha_0,\beta_0)\in\cP(\hhA)$. It then follows from \eqref{mudef4} and Proposition~\ref{muprop} that 
\beqn\label{bx221}
\hat\bx(\alpha_0,\beta_0)=\hat\bx(\mu(\alpha_0),\hat\beta(\alpha_0)-\beta_0)
\eeqn
and, thus, that $\hat\bx$ also maps $(\mu(\alpha_0),\hat\beta(\alpha_0)-\beta_0)$ to $\bx_0$. Consider a point $(\alpha,\beta)\in\cP(\hhA)$ that is mapped to $\bx_0$ under $\hat\bx$. The ruling $[\bc(\alpha),\bc(\mu(\alpha))]$ intersects the ruling $[\bc(\alpha_0),\bc(\mu(\alpha_0))]$ at $\bx_0$. However, by \emph{\ref{M2}} and Proposition~\ref{M2strong}, that intersection is possible only if the two rulings coincide and, hence, only if $\alpha=\alpha_0$ or $\alpha=\mu(\alpha_0)$. If $\alpha=\alpha_0$, then it follows from \eqref{hxthm} that $\beta=\beta_0$; alternatively, if $\alpha=\mu(\alpha_0)$, then $\beta=\hat\beta(\alpha_0)-\beta_0$. Thus, $\hat\bx$ is two-to-one on $\hat\cA$, as claimed. 

Let the set $\hat\br(\cP(\hhA))$ be denoted by $\hat\cS_c$. An argument similar to that given immediately above leads to the conclusion that $\hat\br$ must satisfy
\beqn\label{br221}
\hat\br(\alpha,\beta)=\hat\br(\mu(\alpha),\hat\beta(\alpha)-\beta),\qquad (\alpha,\beta)\in\cP(\hhA).
\eeqn
However, $\hat\br$ is not necessarily two-to-one.

Since \eqref{bx221} and \eqref{br221} hold, it is legitimate to define the function $\bchi:\hat\cD_c\rightarrow\hat\cS_c$ by the provision
\beqn\label{bchidefthm}
\hat\br(\alpha,\beta)=\bchi(\hat\bx(\alpha,\beta)),\qquad (\alpha,\beta)\in\cP(\hhA).
\eeqn
Moreover, since, by Proposition~\ref{locinv},
\beqn
|\hat\bx_\alpha(\alpha,\beta)\times\hat\bx_\beta(\alpha,\beta)|=g_\bm(\alpha)-\beta G(\alpha)\ne0,
\qquad 
(\alpha,\beta)\in \cP(\hat\cA),
\eeqn
the inverse function theorem\footnote{See, for example, Spivak \cite[Chapter 2]{Spivak65}.} can be applied to show that $\hat\bx$ is locally invertible with a $C^1$ inverse. Hence, locally, $\bchi=\hat\br\circ\hat\bx^{-1}$ and, consequently, $\bchi$ is $C^1$ on $\hat\cD_c$. 

To demonstrate that $\bchi$ is an isometry on $\hat\cD_c$, notice first that, by Proposition~\ref{idcond}, the conditions \eqref{parnorm} and \eqref{parcond} can be used to show that, for all $(\alpha,\beta)\in\cP(\hat\cA)$ and $(u_1,u_2)\in\Real^2$,
\beqn
|\nabla\hat\br(\alpha,\beta)(u_1,u_2)|=|\nabla\hat\bx(\alpha,\beta)(u_1,u_2)|.
\eeqn
Upon applying the gradient to \eqref{bchidefthm}, it thus follows that
\beqn
|\nabla\hat\bx(\alpha,\beta)(u_1,u_2)|=|\nabla\bchi(\hat\bx(\alpha,\beta))\nabla\hat\bx(\alpha,\beta)(u_1,u_2)|
\label{437}
\eeqn
for all $(\alpha,\beta)\in\cP(\hat\cA)$ and $(u_1,u_2)\in\Real^2$. Since $\nabla\hat\bx$ is invertible, the condition \eqref{437} is sufficient to ensure that $\nabla\bchi$ preserves the magnitude of every vector tangent to $\hat\cD_c$ and, hence, that $\bchi$ is an isometric immersion from $\hat\cD_c$ to $\hat\cS_c$. 

Since $\bchi$ is a $C^1$ isometry on $\hat{\cD}_c$, it follows that $\nabla\bchi(\hat\bx(\alpha,0))$ must take $\bT(\alpha)$ to $\bt(\alpha)$ and $\bM(\alpha)$ to $\bm(\alpha)$. Hence, $\nabla\bchi(\hat\bx(\alpha,0))=\bt(\alpha)\otimes\bT(\alpha)+\bm(\alpha)\otimes\bM(\alpha)$. Moreover, from the definition of $\bchi$ it can be seen that it takes reference rulings to image rulings, with the consequence that its gradient is constant along rulings. Thus, $\nabla\bchi$ is given by
\beqn\label{gradchi}
\nabla\bchi(\hat\bx(\alpha,\beta))=\bt(\alpha)\otimes\bT(\alpha)+\bm(\alpha)\otimes\bM(\alpha),\qquad 
(\alpha,\beta)\in\cP(\hat\cA).
\eeqn

\end{proofpart}

%\emph{Step 2}: 

\begin{proofpart}

The aim in this step of the proof is to extend $\bff$, $\bg$, $\hat\beta$, and $\mu$ to the boundary of $\hat\cA$. Consider a maximal interval $[\alpha_-,\alpha_+]$ in $\cM\setminus \hhA$, allowing for the possibility that $\alpha_-=\alpha_+$. Focusing first on $\alpha_-$, consider the question of whether the limit
\beqn\label{flimit}
\lim_{\alpha\uparrow\alpha_-,\,\alpha\in\hat\cA}\bff(\alpha)
\eeqn
exists. If the limit in \eqref{flimit} does not exist, then, since $\bff$ has a compact codomain, there must exist increasing sequences $\alpha_{1n}\in\hat\cA$ and $\alpha_{2n}\in\hat\cA$, both converging to $\alpha_-$, with $\alpha_{1n}<\alpha_-$ and $\alpha_{2n}<\alpha_-$, but for which\beqn
\lim_{n\rightarrow\infty}\bff(\alpha_{1n})\ne\lim_{n\rightarrow\infty}\bff(\alpha_{2n}).
\eeqn
This implies that there would be rulings that are arbitrarily close but point in different directions. Such rulings must necessarily cross, violating \emph{\ref{M2}}. The limit in \eqref{flimit} therefore exists, with the consequence \eqref{defgg} and \eqref{frepp} can be used to conclude that the analogous limit in which $\bff$ is replaced by $\bg$ exists. It is thus possible to define $\bff_-(\alpha_-)$ and $\bg_-(\alpha_-)$ by
\beqn\label{rlimits}
\bff_-(\alpha_-)=\lim_{\alpha\uparrow\alpha_-,\,\alpha\in\hhA}\bff(\alpha)\qquad\text{and}\qquad
\bg_-(\alpha_-)=\lim_{\alpha\uparrow\alpha_-,\,\alpha\in\hhA}\bg(\alpha).
\eeqn
Since $\mu$ is decreasing on $\hat\cA$ and is bounded from below at $\alpha_-$, it is also possible to define $\mu_-(\alpha_-)\in\cM$ by
\beqn\label{mulimit}
\mu_-(\alpha_-)=\lim_{\alpha\uparrow\alpha_-,\,\alpha\in\hhA}\mu(\alpha).
\eeqn

Next, since
\beqn
\bff(\alpha)=\frac{\bc(\mu(\alpha))-\bc(\alpha)}{|\bc(\mu(\alpha))-\bc(\alpha)|},\qquad\alpha\in\hhA,
\eeqn
it follows from \eqref{rlimits}$_1$ that if $\mu_-(\alpha_-)\not=\alpha_-$, then
\beqn
\bff_-(\alpha_-)=\lim_{\alpha\uparrow\alpha_-,\,\alpha\in\hhA}\frac{\bc(\mu(\alpha))-\bc(\alpha)}{|\bc(\mu(\alpha))-\bc(\alpha)|}=\frac{\bc(\mu_-(\alpha_-))-\bc(\alpha_-)}{|\bc(\mu_-(\alpha_-))-\bc(\alpha_-)|}
\eeqn
and, thus, that $\bc(\mu_-(\alpha_-))$ must be on the half-line emanating from $\bc(\alpha_-)$ in the direction $\bff_-(\alpha_-)$. Hence, $\mu_-(\alpha_-)\not\in\hhA$. Moreover, $\mu_-(\alpha_-)$ must be the right endpoint of a maximal interval in $\cM\setminus\hhA$. An analogous argument can be used to demonstrate that $\bd(\mu_-(\alpha_-))$ must be on the half-line that emanates from $\bd(\alpha_-)$ and is directed along $\bg_-(\alpha_-)$. Thus, in view of \eqref{mulimit} and the relation
\beqn
\hat\beta(\alpha)=|\bc(\mu(\alpha))-\bc(\alpha)|=|\bd(\mu(\alpha))-\bd(\alpha)|,\qquad\alpha\in\hhA,
\eeqn
it is possible to define $\hat\beta_-(\alpha_-)$ by
\beqn
\hat\beta_-(\alpha_-)=\lim_{\alpha\uparrow\alpha_-,\,\alpha\in\hhA}\hat\beta(\alpha).
\eeqn
In view of \eqref{mulimit}, it thus follows that
\beqn
\bc(\mu_-(\alpha_-))=\bc(\alpha_-)+\hat\beta_-(\alpha_-)\bff_-(\alpha_-)
\eeqn
and that
\beqn
\bd(\mu_-(\alpha_-))=\bd(\alpha_-)+\hat\beta_-(\alpha_-)\bg_-(\alpha_-).
\eeqn

An analogous argument can be applied to the right endpoint $\alpha_+$ of the interval $[\alpha_-,\alpha_+]$.  A subscripted plus sign will therefore be used to denote objects associated with $\alpha_+$, in parallel with the convention already used for objects associated with $\alpha_-$.

Line segments of the form $[\bc(\alpha_-),\bc(\mu_-(\alpha_-))]$ or $[\bc(\alpha_+),\bc(\mu_+(\alpha_+))]$ will be called $\partial$-rulings --- as they are belong to the boundary of the set $\hat\cD_c$ that is covered by rulings. Since rulings cannot cross if \emph{\ref{M1}} holds, $\partial$-rulings cannot cross in $\cD$, but, unlike rulings, they can meet at the boundary $\partial\cD$, overlap with each other on straight-line segments, or have zero length. See Figure~\ref{Dfig2} for a depiction of $\partial$-rulings.
Each $\partial$-ruling $\cL_0$ belongs to one of three mutually exclusive categories, determined by the properties of its endpoints:

\begin{enumerate}[label=\textbf{\textit{L\arabic*}}]

\item \label{L1} there is an endpoint $\bc(\alpha_0)$ of $\cL_0$  such that $\alpha_0\not\in\cF$, $\bn'(\alpha_0)\ne\bzero$, and $\cL_0$ is not tangent to $\partial\cD$ at $\bc(\alpha_0)$,

\item \label{L2} there is an endpoint $\bc(\alpha_0)$ of $\cL_0$ such that $\alpha_0\not\in\cF$, $\bn'(\alpha_0)=\bzero$, and $\cL_0$ is not tangent to $\partial\cD$ at $\bc(\alpha_0)$,

\item \label{L3} at each endpoint $\bp$ of $\cL_0$, $\partial\cD$ has a corner at $\bp$, $\cL_0$ is tangent to $\partial\cD$ at $\bp$, or both of these conditions are met.

\end{enumerate}

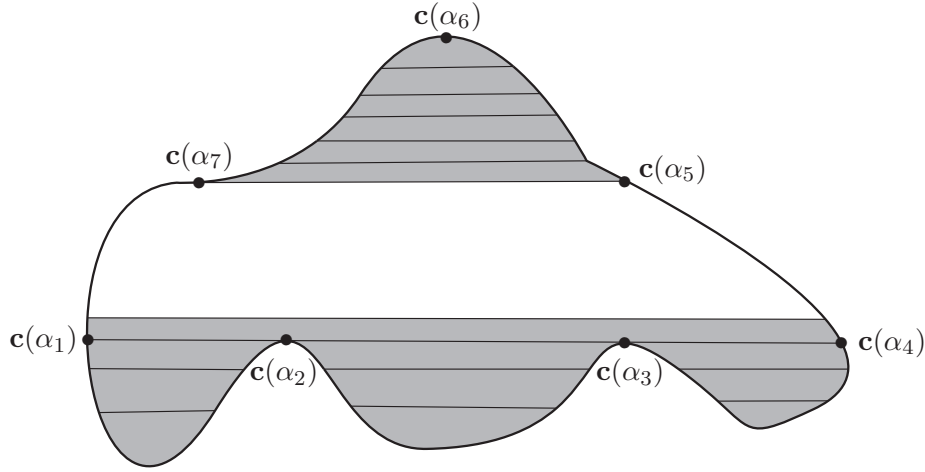
\begin{figure}[h]
\centering
\includegraphics[width=4in]{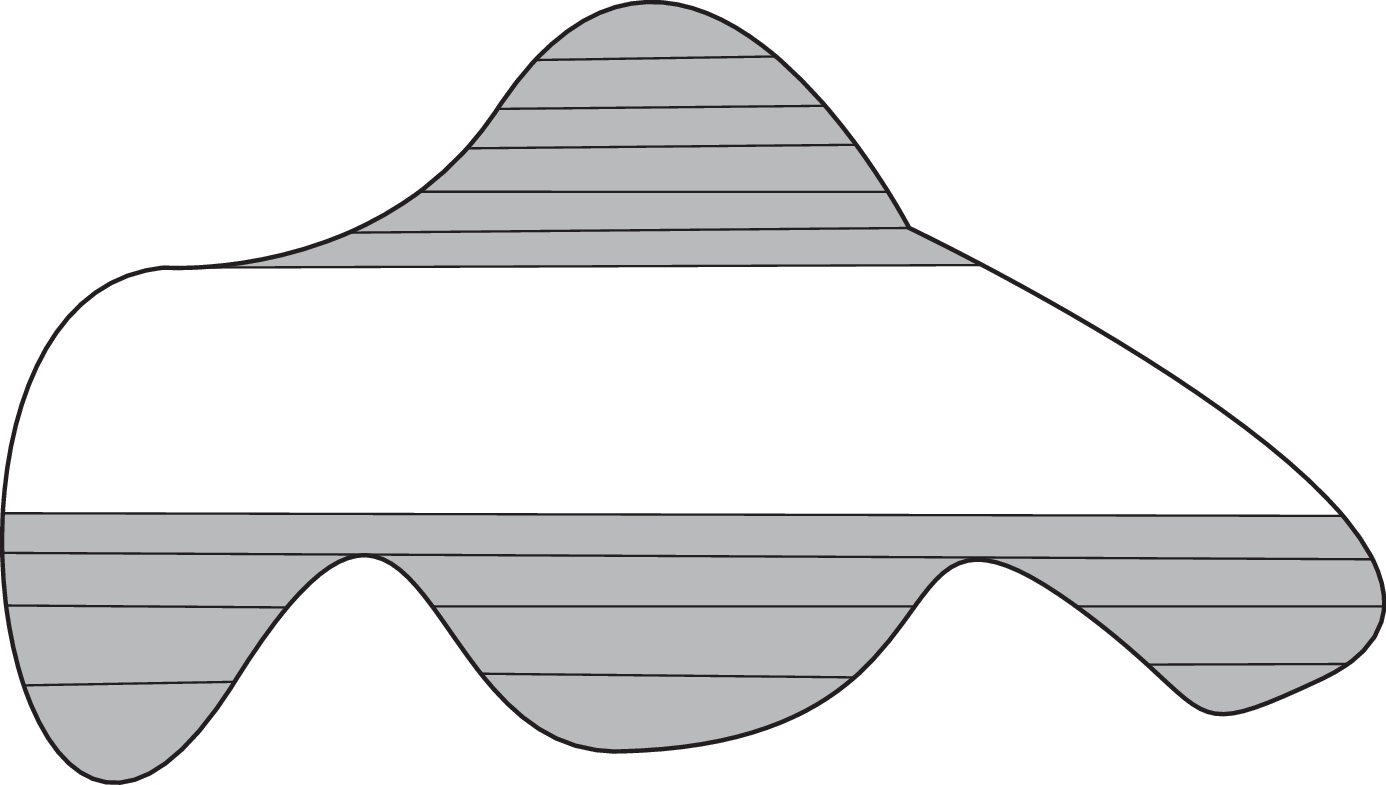}
\thicklines
\put(-291,45.5){$\bullet$}
\put(-318,45.5){$\bc(\alpha_1)$}
\put(-216,45.5){$\bullet$}
\put(-227,33){$\bc(\alpha_2)$}
\put(-88,44.5){$\bullet$}
\put(-96,33){$\bc(\alpha_3)$}
\put(-6,44.5){$\bullet$}
\put(3,44.5){$\bc(\alpha_4)$}
\put(-249,105){$\bullet$}
\put(-259,115){$\bc(\alpha_7)$}
\put(-88,105.5){$\bullet$}
\put(-80,110){$\bc(\alpha_5)$}
\put(-155.5,160){$\bullet$}
\put(-165,168){$\bc(\alpha_6)$}
\caption{Reference region $\cD$ with straight line segments representing rulings and $\partial$-rulings induced by a framed curve. The grey zones correspond to the subset $\hat\cD_c$ of $\cD$, up to a set of measure zero. For instance, the line segment $[\bc(\alpha_1),\bc(\alpha_4)]$ is not contained in $\hat\cD_c$. Assuming that $\bc$ parametrizes $\partial\cD$ in the counterclockwise direction, the $\partial$-ruling $[\bc_-(\alpha_1),\bc_-(\mu_-(\alpha_1))]$ coincides with the line segment $[\bc(\alpha_1),\bc(\alpha_4)]$, while the $\partial$-ruling $[\bc_+(\alpha_1),\bc_+(\mu_-(\alpha_1))]$ equals $[\bc(\alpha_1),\bc(\alpha_2)]$. The line segment $[\bc(\alpha_2),\bc(\alpha_3)]$ corresponds to the $\partial$-ruling $[\bc_+(\alpha_2),\bc_+(\mu_+(\alpha_2))]$, and the $\partial$-ruling $[\bc_+(\alpha_5),\bc_+(\mu_+(\alpha_5))]$ equals $[\bc(\alpha_5),\bc(\alpha_7)]$. Finally, the point $\bc(\alpha_5)$ coincides with the degenerate $\partial$-ruling $[\bc_-(\alpha_6),\bc_-(\mu_-(\alpha_6))]$, which has zero length.}
\label{Dfig2} 
\end{figure}

Before moving to the next step in the proof, it is shown that a $\partial$-ruling that has positive length and satisfies \emph{\ref{L3}} is isolated in the sense that excluding its endpoints, it is contained in a neighborhood of $\cD$ that intersects no other $\partial$-ruling of positive length that also satisfies \emph{\ref{L3}}. To establish this, consider a $\partial$-ruling $\cL_0$ with positive length. It is first shown that for every $\bx_0\in\cL_0\cap\cD$ there is a neighborhood that does not intersect any other $\partial$-ruling satisfying \emph{\ref{L3}}. This conclusion is reached by the method of contradiction. If $\bx_0$ were not isolated as described, there would exist a sequence of points $\bx_n\in\cD$ converging to $\bx_0$ such that $\bx_n$ lies on a $\partial$-ruling, say $\cL_n$, distinct from $\cL_0$, that satisfies \emph{\ref{L3}}. Since $\partial\cD$ has only a finite number of corners, for sufficiently large $n$ the endpoints of $\cL_n$ cannot be corners of $\partial\cD$ and, so, must be points of tangency with $\partial\cD$. For $n\in\Nat\cup\{0\}$, let $\alpha_n\in\cM$ and $\alpha_n^*\in\cM$ be such that $\cL_n=[\bc(\alpha_n),\bc(\alpha_n^*)]$. Since no two $\partial$-rulings can cross in $\cD$, it can be assumed, by potentially considering a subsequence, that the order of the $\partial$-rulings in this sequence are such that $\alpha_n$ is a decreasing sequence with $\alpha_0<\alpha_n$ and that $\alpha_n^*$ is an increasing sequence with $\alpha^*_n<\alpha^*_0$. It is then evident that the limits
\beqn\label{alphanlim}
\lim_{n\rightarrow\infty}\alpha_n=\alpha_\infty
\qquad\text{and}\qquad 
\lim_{n\rightarrow\infty}\alpha^*_n=\alpha^*_\infty
\eeqn
exist. However, since $\cL_n$ is tangent to $\partial\cD$ at $\bc(\alpha_n)$ for large $n$, it must be true that $\kappa_\bc(\alpha_n)\le 0$ for such $n$. Moreover, since a $\partial$-ruling cannot have a tangent endpoint in the interior of an interval on which $\kappa_\bc$ vanishes and since $\kappa_\bc$ can only vanish on a finite number of intervals, it follows that, for large $n$, $\kappa_\bc(\alpha_n)<0$ and that, between $\alpha_n$ and $\alpha_{n+1}$, $\kappa_\bc< 0$. Hence, the vector $\bc(\alpha_n^*)-\bc(\alpha_n)$ must be directed away from the line segment $[\bc(\alpha_\infty),\bc(\alpha_\infty^*)]$ for large $n$. The same observation applies to $\alpha_n^*$. Thus, the vector $\bc(\alpha_n)-\bc(\alpha_n^*)$ is also directed away from $[\bc(\alpha_\infty),\bc(\alpha_\infty^*)]$ for large $n$. However, this is not possible and, consequently, there is a neighborhood of $\bx_0$ that is devoid of any $\partial$-rulings that satisfy \emph{\ref{L3}}. As this is true for every $\bx_0\in\cL_0\cap\cD$, taking the union of all such neighborhoods over all possible such $\bx_0$ yields a neighborhood of $\cL_0$ relative to $\cD$ that does not intersect any other $\partial$-rulings satisfying \emph{\ref{L3}}. 

The above argument also leads to the conclusion that the set of all points in $\cD$ on $\partial$-rulings satisfying \emph{\ref{L3}} is a closed set relative to $\cD$. To see that this is the case, consider a sequence of points $\bx_n\in\cD$ that are on $\partial$-rulings satisfying \emph{\ref{L3}} that converge to some $\bx_0\in\cD$. If $\bx_n$ are all on the same $\partial$-ruling for large enough $n$, then it must be true that $\bx_0$ is also on this $\partial$-ruling, which establishes the claim. If not, then the argument in the previous paragraph can be repeated to reach a contradiction.

\end{proofpart}

%\emph{Step 3}: 

\begin{proofpart}

The aim in this part of the proof is to show that the connected components of $\bar\cD\setminus\hat\cD_c$ can be rigidly rotated and translated to fill any gaps present in $\hat\cS_c\cup\bd(\cM)$. Consider, once again, a maximal interval $[\alpha_-,\alpha_+]$ in $\cM\setminus\hat\cA$. To begin, consider the question of whether the set $\bc([\alpha_-,\alpha_+])$ can be rigidly transformed to the set $\bd([\alpha_-,\alpha_+])$. If $\alpha_-=\alpha_+$, then $\bc([\alpha_-,\alpha_+])$ and $\bd([\alpha_-,\alpha_+])$ are both singletons and thus differ by a translation. Next, suppose that $\alpha_-\ne\alpha_+$. Since $[\alpha_-,\alpha_+]$ is disjoint from $\hat\cA$, $\bn$ must satisfy $\bn'=\bzero$ almost everywhere on $[\alpha_-,\alpha_+]$ for the reduced bending energy \eqref{RedBEmod}--\eqref{RedED} to be finite. Since $\bn$ is continuous, this implies that $\bn$ must be constant on $[\alpha_-,\alpha_+]$ and, hence, that the curve parametrized by $\bd$ must be planar on $[\alpha_-,\alpha_+]$.
Moreover, by \eqref{cdncond}, the geodesic curvature and corner angles of the curve parametrized by $\bd$ match those of the curve parametrized by $\bc$. Thus, $\bc([\alpha_-,\alpha_+])$ can be rigidly transformed to $\bd([\alpha_-,\alpha_+])$.

Now consider a connected component $\cD_f$ of $\bar\cD\setminus\hat\cD_c$. The goal is to show that there is a gap in $\hat\cS_c\cup\bd(\cM)$ that has the exact same shape as $\cD_f$. The boundary $\partial\cD_f$ of $\cD_f$ must contain parts of $\partial\cD$ to avoid violating \emph{\ref{M2}}. Let $\cM_f\subseteq\cM$ denote the collection of $\alpha\in\cM$ for which $\bc(\alpha)\in\partial\cD_f$. From the argument presented in the previous paragraph, it is known that $\bc(\cI)$ and $\bd(\cI)$ are related by a rigid transformation for every maximal interval $\cI$ in $\cM_f$. What remains of $\partial\cD_f$ is contained in $\cD$. Each point of $\partial\cD_f$ that does not belong to $\partial\cD$ must lie on a straight line segment with endpoints on $\bc(\cM_f)$. Thus, every such line segment is a $\partial$-ruling of the form $[\bc(\alpha_-),\bc(\mu_-(\alpha_-))]$, where $[\alpha_-,\alpha_+]$ is some maximal interval in $\cM_f$. Moreover, since it has already been established that $\hat\beta_-(\alpha_-)=|\bc(\mu_-(\alpha_-)-\bc(\alpha_-)|=|\bd(\mu_-(\alpha_-)-\bd(\alpha_-)|$, it can be concluded that there is a gap in $\hat\cS_c\cup\bd(\cM)$ with boundary consisting of pieces exactly the same size and shape as the boundary of $\cD_f$. It remains to be show that this gap is planar. 

Towards this end, recall that
\beqn\label{nrelthm}
\bn=\frac{\bd'\times\bg}{|\bd'\times\bg|}\qquad\text{on}\ \hhA.
\eeqn
Differentiating the condition in \emph{\ref{M1}} and forming its vector product with $\bg(\mu(\alpha))$ yields the identity
\begin{align}
\mu'(\alpha)\bd'(\mu(\alpha))\times\bg(\mu(\alpha))&=-(\bd'(\alpha)+\hat\beta(\alpha)\bg'(\alpha))\times\bg(\alpha)
\notag\\[4pt]
&=-(g_\bm(\alpha)-\hat\beta(\alpha)G(\alpha))\bn(\alpha).
\label{obtainingnrel}
\end{align}
By \eqref{obtainingnrel} and Proposition \ref{locinv},
\beqn\label{nonzero}
|\mu'(\alpha)\bd'(\mu(\alpha))\times\bg(\mu(\alpha))|
=|g_\bm(\alpha)-\hat\beta(\alpha)G(\alpha)|\ne0.
\eeqn
On dividing both sides of \eqref{obtainingnrel} by $|g_\bm(\alpha)-\hat\beta(\alpha)G(\alpha)|\ne0$ and using \eqref{nrelthm} and Proposition~\ref{bbetamureg}, it follows that $\bn(\alpha)=\bn(\mu(\alpha))$ for $\alpha\in\hat\cA$. Thus, if $[\alpha_-,\alpha_+]$ is any maximal interval in $\cM_f$, using \eqref{mulimit}, yields
\beqn
\bn(\alpha_-)=\bn(\mu_-(\alpha_-)).
\eeqn
Next, consider any $\alpha,\alpha'\in\cM_f$ and the path along the boundary of $\cD_f$ that connects $\bc(\alpha)$ to $\bc(\alpha')$. There is a corresponding path connecting $\bd(\alpha)$ to $\bd(\alpha')$ that consists of portions of $\bd(\cM_f)$ and of image $\partial$-rulings. Along each of these sets the normal $\bn$ is constant. Moreover, the path in question is connected and, thus, must be contained in a plane. Thus, $\bd(\cM_f)$ is planar. It follows that it is possible to extend the definition of $\bchi$ so that it rigidly transforms $\cD_f$ to fill the gap in $\hat\cS_c\cup\bd(\cM)$ with a gap, necessarily of identical size, which includes $\bd(\cM_f)$ as part of its boundary. This can be done for each connected component in $\bar\cD\setminus\hat\cD_c$, and the resulting function $\bchi$ defined on $\bar\cD$ is continuous by construction.

\end{proofpart}

%\emph{Step 4}: 

\begin{proofpart}

The aim of this part of the proof is to show that $\bchi$ is $C^1$ in $\cD$. Let $\bx_0\in\cD$. If $\bx_0$ belongs to $\hat\cD_c$, then $\bchi$ is $C^1$ in a neighborhood of $\bx_0$, as mentioned in \emph{Step}~1. Moreover, if $\bx_0$ belongs to the interior of any of the connected components of $\bar\cD\setminus\hat\cD_c$ mentioned in \emph{Step}~3, then $\bchi$ is also $C^1$ in a neighborhood of $\bx_0$, since those components are rigidly transformed by $\bchi$. Next, suppose that $\bx_0$ does not belong to either of these sets. From \emph{Step}~3, it follows that $\bx_0$ must lie on a $\partial$-ruling, say $\cL_0$, with endpoints on $\partial\cD$. The three mutually exclusive cases \emph{\ref{L1}--\ref{L3}} for $\cL_0$ will be considered separately.

First, suppose that $\cL_0$ satisfies \emph{\ref{L1}}. It follows from \eqref{cdncond2} that $\kappa_n(\alpha_0)\ne0$ and, thus, that $\alpha_0\in\tilde\cA$. Moreover, $\cL_0=[\bc(\alpha_0),\bc(\mu(\alpha_0))]$. Consider an interval $\cI\subseteq\cM$ chosen so $\alpha_0\in\cI$ and $\bn'\ne\bzero$ on $\cI$. It is then possible to define parameterizations $\hat\bx$ and $\hat\br$ as in \eqref{hxthm} and \eqref{hrthm}, respectively, for $(\alpha,\beta)\in\cP(\cI)$. To ensure that the reduced bending energy \eqref{RedBEmod} is finite, almost every point in $\cI$ must also belong to $\hat\cA$. By combining this observation with the continuity of $\hat\bx$, $\hat\br$, and $\bchi$, it becomes evident that $\hat\bx$ and $\hat\br$ determine the function $\bchi$ on $\hat\bx(\cP(\cI))$. The argument in \emph{Step}~1 can now be repeated to show that $\bchi$ is $C^1$ in a neighborhood of $\bx_0$. Additionally, the representation \eqref{gradchi} for the gradient of $\bchi$ also holds for $\alpha\in\tilde\cA$.

Next, suppose that $\cL_0$ satisfies \emph{\ref{L2}}. Then, there is a unit vector $\bff_0$ parallel to $\cL_0$ and a positive number $\beta_0$ such that
\beqn
\bx_0=\bc(\alpha_0)+\beta_0\bff_0.
\eeqn
Let $\bQ$ be defined such that
\beqn\label{bQrep}
\bQ(\alpha)=\bt(\alpha)\otimes\bT(\alpha)+\bm(\alpha)\otimes\bM(\alpha),
\qquad \alpha\in\cM\setminus\cF.
\eeqn
Recalling the final consequence \eqref{gradchi} of \emph{Part}~1, it can be shown that $\nabla\bchi(\bx_0)=\bQ(\alpha_0)$. This will be accomplished by showing that
\beqn\label{defdiff}
\lim_{n\rightarrow\infty}\frac{|\bchi(\bx_n)-\bchi(\bx_0)-\bQ(\alpha_0)(\bx_n-\bx_0)|}{|\bx_n-\bx_0|}=\bzero
\eeqn
for any sequence $\bx_n$ that converges to $\bx_0$. Towards this end, let $\bx_n$ be such a sequence. Each point $\bx_n$ is either on a ruling, a $\partial$-ruling, or belongs to $\cD\setminus\hat\cD_c$. Regardless, the representation
\beqn
\bx_n=\bc(\alpha_n)+\beta_n\bff_n
\eeqn
is valid for some $\alpha_n\in\cM$, $\beta_n>0$, and $|\bff_n|=1$  such that the line segment $[\bc(\alpha_n),\bx_n]$ is part of either a ruling, $\partial$-ruling, or is contained in a connected component of $\bar\cD\setminus\hat{\cD}_c$. Since $\cL_0$ is not tangent to $\partial\cD$ at $\bx_0$, the $\alpha_n$ can be chosen so that 
\beqn
\lim_{n\rightarrow\infty}\alpha_n=\alpha_0,
\eeqn
and clearly $\displaystyle\lim_{n\rightarrow\infty}\beta_n=\beta_0$ and $\displaystyle\lim_{n\rightarrow\infty}\bff_n=\bff_0$. Since $\bchi$ takes referential rulings and $\partial$-rulings to image rulings and image $\partial$-rulings, respectively, and since $\bQ$, as defined in \eqref{bQrep}, maps the direction of referential rulings and $\partial$-rulings to the direction of image rulings and image $\partial$-rulings, it follows that
\begin{align}
\bchi(\bx_n)&=\bd(\alpha_n)+\beta_n\bQ(\alpha_n)\bff_n,\quad n\in\Nat\cup\{0\}.\label{ynrep}
\end{align}
Thus, if $\alpha_n=\alpha_0$ for any $n$, then the ratio $|\bchi(\bx_n)-\bchi(\bx_0)-\bQ(\alpha_0)(\bx_n-\bx_0)|/|\bx_n-\bx_0|$ that appears in \eqref{defdiff} vanishes. It can therefore, henceforth, be assumed that $\alpha_n\ne\alpha_0$. A calculation using \eqref{DFr} and \eqref{DFs} can then be used to show that
\beqn\label{bQprep}
\bQ'=-\bn\otimes\bQ\bn'.
\eeqn
Since $\bn'(\alpha_0)=\bzero$, it follows that
\beqn\label{Qpz}
\bQ(\alpha_n)=\bQ(\alpha_0)+o(\alpha_n-\alpha_0),
\eeqn
where a term of $o(\zeta)$ vanishes faster than $\zeta$ as $\zeta\to0$. On combining \eqref{ynrep} and \eqref{Qpz} and leveraging the understanding that $\bc$ and $\bd$ are differentiable at $\alpha_0$, it follows that
\beqn\label{diffnumlim}
|\bchi(\bx_n)-\bchi(\bx_0)-\bQ(\alpha_0)(\bx_n-\bx_0)|=o(\alpha_n-\alpha_0).
\eeqn
Consider the ratio
\beqn
R_n=\frac{|\alpha_n-\alpha_0|}{|\bc(\alpha_n)-\bc(\alpha_0)+\beta_n\bff_n-\beta_0\bff_0|},
\eeqn
and let $\bff^\perp$ be defined by
\beqn
\bff^\perp_0=\bff_0\times\bN(\alpha_0).
\eeqn
The next objective is to establish an upper bound on $R_n$, independent of $n$. If $\bff_n\ne\bff_0$, then there are numbers $\gamma_n$ and $\gamma_0$ such that
\beqn
\bc(\alpha_0)+\gamma_0\bff_0=\bc(\alpha_n)+\gamma_n\bff_n.
\label{460}
\eeqn
Observe that $\gamma_n\not=0$ for large $n$. Considering the component of \eqref{460} along $\bff^\perp_0$ leads to the relation
\beqn
\bff_n\cdot\bff_0=\frac{\bff_0^\perp\cdot(\bc(\alpha_0)-\bc(\alpha_n))}{\gamma_n}.
\label{461}
\eeqn
Thus, $R_n$ has the upper bound
\begin{align}
R_n&\le \frac{|\alpha_n-\alpha_0|}{|\bff_0^\perp\cdot(\bc(\alpha_n)-\bc(\alpha_0)+\beta_n\bff_n-\beta_0\bff_0)|}\notag\\[4pt]
&=\frac{|\alpha_n-\alpha_0|}{|(1-\beta_n/\gamma_n)\bff_0^\perp\cdot((\bc(\alpha_n)-\bc(\alpha_0))|}.\label{Rnbound}
\end{align}
The bound \eqref{Rnbound} holds if $\bff_n\to\bff_0$ and $\gamma_n\to\infty$. The term on the right-hand side of \eqref{Rnbound} is bounded uniformly in $n$ for sufficiently large $n$ since: (i) $\bc$ is differentiable at $\alpha_0$; (ii) $\bc'(\alpha_0)$ is not parallel to $\bff_0$; and, (iii) $\gamma_n\ne\beta_n$ for large $n$, as otherwise rulings or $\partial$-rulings would cross in $\cD$. It can be seen that $R_n$ is uniformly bounded for those $n$ such that $\bff_n=\bff_0$ as $|\bc'(\alpha_0)=1$. Combining \eqref{diffnumlim} with the uniform bound on $R_n$ yields \eqref{defdiff}. Thus, $\bchi$ is differentiable at $\bx_0$, as claimed. Moreover, since $\nabla\bchi(\bx_0)=\bQ(\alpha_0)$, $\bQ$ is continuous, and \eqref{gradchi} holds, it follows that $\bchi$ is $C^1$ in a neighborhood of $\bx_0$.

Finally, suppose that $\cL_0$ satisfies \emph{\ref{L3}}. Since $\cL_0$ contains a point of $\cD$, it has positive length. Thus, from the argument at the end of \emph{Part}~2, it is isolated from other $\partial$-rulings satisfying \emph{\ref{L3}}. From the arguments appearing in the previous two paragraphs, there exist neighborhoods of $\bx_0$, on either side of $\cL_0$, within which $\bchi$ has $C^1$ regularity. Thus, near $\bchi(\bx_0)$, the surface $\bchi(\cD)$ is formed by two $C^1$ surfaces whose boundaries meet along the line segment corresponding to the image $\partial$-ruling $\bchi(\cL_0)$. These surfaces with common boundary must be oriented so that unit normal field is continuous across the junction $\cL_0$ as the normal field is determined by the normal along the bounding curve parametrized by $\bd$. Hence, the resulting surface is $C^1$ in a neighborhood of $\bchi(\bx_0)$.

Since $\bchi$ is an isometry and maps a $C^1$ surface to a $C^1$ surface, a result due to Myers and Steenrod \cite{MS39} can be used to show that $\bchi$ is $C^1$ in a neighborhood of $\bx_0$. As $\bx_0\in\cD$ can be chosen arbitrarily, it follows that $\bchi$ is $C^1$ on $\cD$ and, thus, that $\bchi(\cD)$ is a $C^1$ immersed surface. 

\end{proofpart}

%\emph{Step 5}: 

\begin{proofpart}

The aim of this part of the proof is to show that $\bchi$ has a $C^1$ extension to the closed set $\bar\cD$. From \emph{Steps~1--4}, it is known that for $\bx\in\cD$, $\nabla\bchi(\bx)=\bQ(\alpha)$ (see \eqref{bQrep}), where $[\bc(\alpha),\bx]$ is a line segment that is either part of a ruling, a $\partial$-ruling, or contained in a planar part of $\cD$. Since $\bQ(\alpha)$ is constant on $[\bc(\alpha),\bx]$, the gradient $\nabla\bchi$ extends continuously to $\partial\cD$. From \eqref{bQrep} it can be seen that $\bQ$ is a continuous function away from $\cF$. However, as $\bQ(\alpha)$ is the rotation that takes $\bT(\alpha)$ to $\bt(\alpha)$ and $\bM(\alpha)$ to $\bm(\alpha)$ and, moreover, condition \eqref{cdncond}$_2$ holds, it follows that $\bQ$ is continuous on all of $\cM$. This ensures that the extension of $\nabla\bchi$ to the closed set $\bar\cD$ is continuous. Thus, $\bchi$ is $C^1$ on $\bar\cD$.
\end{proofpart}

%\emph{Step 6}: 

\begin{proofpart}
This final part of the proof is dedicated to showing that $\bchi$ is $C^2$ almost everywhere. Begin by considering the unit normal $\hat\bnu$ on $\cS$ pulled back to $\cD$ using $\bchi$, meaning that for $\bx\in\cD$, $\hat\bnu(\bx)$ gives the normal to $\cS$ at $\bchi(\bx)$.\footnote{See the discussion in the paragraph before the start of Subsection~\ref{sectpar}.} It will be shown that on an open subset $\cD_s$ of $\cD$ that differs from $\cD$ by a set of zero areal measure, $\hat\bnu$ is $C^1$. It then follows that $\bchi(\cD_s)$ is locally a $C^2$ surface. As at the end of \emph{Part}~4, the result of Myers and Steenrod \cite{MS39} can then be applied to show that the restriction of $\bchi$ to $\cD_s$ is $C^2$.

From the previous parts of the proof it is known that $\cD$ consists of three types of subsets: the part $\hat\cD_c$ of $\cD$ that is described by the parameterization $\hat\bx$ defined in \eqref{hrthm}, the part of $\cD$ that consists of $\partial$-rulings, and the remaining part of $\cD$, which consists of open pieces that remain planar under $\bchi$. On $\hat\cD_c$, we have $\hat\bnu\circ\hat\bx=\bnu$, where $\bnu$ is defined in \eqref{bnu}. Upon taking the gradient of this equation and evaluating it at $(\alpha,\beta)\in\cP(\hat\cA)$, it follows that
\beqn
\nabla \hat\bnu(\hat\bx(\alpha,\beta))=\nabla\bnu(\alpha,\beta)(\nabla\hat\bx(\alpha,\beta))^{-1}.
\eeqn
Using \eqref{hxthm}, a calculation shows that for any vector $\bw$ tangent to $\cD$,
\beqn
(\nabla\hat\bx(\alpha,\beta))^{-1}\bw=\Big (\frac{\bff^\perp(\alpha)\cdot\bw}{g_\bm(\alpha)-\beta G(\alpha)},\bff(\alpha)\cdot\bw-\frac{\bff^\perp(\alpha)\cdot\bw}{g_\bm(\alpha)-\beta G(\alpha)}\Big),
\eeqn
where $\bff^\perp=g_\bm\bT-g_\bt\bM$ is orthogonal to $\bff$ and tangent to $\cD$. Thus,
\beqn\label{gradhbnu}
\nabla\hat\bnu(\hat\bx(\alpha,\beta))=\frac{\bn'(\alpha)\otimes\bff^\perp(\alpha)}{g_\bm(\alpha)-\beta G(\alpha)}.
\eeqn
It follows from \eqref{gradhbnu} that $\hat\bnu$ is $C^1$ on $\hat\cD_c$. Moreover, $\hat\bnu$ is also $C^1$ on the open parts of $\cD$ that remain planar under $\bchi$ as $\hat\bnu$ is locally constant on this set.

Suppose now that $\bx_0\in\cD$ lies on a $\partial$-ruling $\cL_0$. The $\partial$-ruling $\cL_0$ must then belong to one of the three categories \emph{\ref{L1}}--\emph{\ref{L3}} delineated in the penultimate paragraph of \emph{Part}~4. If $\cL_0$ satisfies \emph{\ref{L1}}, then $\bx_0$ is covered by a parameterization $\hat\bx$ of the form \eqref{hxthm} with $\hat\cA$ replaced by $\tilde\cA$. Thus, for the same reason as in the previous paragraph, the gradient of $\hat\bnu$ at $\bx$ also satisfies \eqref{gradhbnu} for some $(\alpha,\beta)\in \cP(\tilde\cA)$. Next, suppose that $\cL_0$ satisfies \emph{\ref{L2}}. Using an argument similar to that appearing in the third paragraph of \emph{Part}~4, it can be shown that $\hat\bnu$ is differentiable at $\bx_0$ and its gradient at $\bx_0$ is zero. Finally, suppose that the $\cL_0$ satisfies \emph{\ref{L3}}. It was argued at the end of \emph{Part}~2 that such rulings are isolated. Thus, there can be at most a countable number of such $\partial$-rulings. It follows that the set of points in $\cD$ that are covered by rulings of this type form a set of zero areal measure which, as argued at the end of \emph{Part}~2, must be closed.

It has been shown that $\nabla\hat\bnu$ exists away from $\partial$-rulings satisfying \emph{\ref{L3}}. Next it is shown that this function is continuous. In particular, it remains to be shown that $\nabla\hat\bnu$ is continuous on $\partial$-rulings satisfying \emph{\ref{L2}} as on those satisfying \emph{\ref{L1}} the formula \eqref{gradhbnu} holds. Let $\bx_0\in\cD$ be a point on $\partial$-ruling $\cL_0$ satisfying \emph{\ref{L2}}. Let $\bx_n\in\cD$ be a sequence of points that converge to $\bx_0$. It must be shown that $\nabla\hat\bnu(\bx_n)$ converges to $\bzero$ as $n$ goes to infinity as $\nabla\hat\bnu(\bx_0)=\bzero$. 
It suffices to consider $\bx_n$ of the form
\beqn
\bx_n=\hat\bx(\alpha_n,\beta_n)=\bc(\alpha_n)+\beta_n\bff(\alpha_n),\qquad (\alpha_n,\beta_n)\in\cP(\tilde\cA).
\eeqn
Since $\cL_0$ is a $\partial$-ruling, there is a maximal interval $[\alpha_-,\alpha_+]\subseteq\cM\setminus\hat\cA$ such that $\cL_0=[\bc(\alpha_-),\bc(\mu_-(\alpha_-))]$. As $\cL_0$ satisfies \emph{\ref{L2}}, it can be assumed that $\bn'(\alpha_-)=\bzero$. Given $\beta_0\in(0,\hat\beta_-(\alpha_-))$ such that
\beqn
\bx_0=\bc(\alpha_-)+\beta_0\bff_-(\alpha_-).
\eeqn
Since $\lim_{n\rightarrow\infty}\bx_n=\bx_0$, it follows that
\beqn\label{limab}
\lim_{n\rightarrow\infty}\alpha_n=\alpha_0\qquad\text{and}\qquad \lim_{n\rightarrow\infty}\beta_n=\beta_0.
\eeqn
Utilizing Proposition~\ref{locinv}, it follows that
\beqn
\inf_{n\in\Nat} (g_m(\alpha_n)-\hat\beta(\alpha_n)G(\alpha_n))\geq 0.
\eeqn
As $\beta_0<\hat\beta_-(\alpha_-)$, the previous inequality and \eqref{limab}$_2$ yield
\beqn
\inf_{n\in\Nat} (g_m(\alpha_n)-\beta_nG(\alpha_n))=:I>0.
\eeqn
Thus, from \eqref{gradhbnu},
\beqn
\lim_{n\rightarrow\infty} |\nabla\hat\bnu(\alpha_n)|\leq\lim_{n\rightarrow\infty}\frac{|\bn'(\alpha_n)|}{g_m(\alpha_n)-\beta_nG(\alpha_n)}\leq \lim_{n\rightarrow\infty}\frac{|\bn'(\alpha_n)|}{I}=\lim_{n\rightarrow\infty}\frac{|\bn'(\alpha_-)|}{I}=0,
\eeqn
as desired and, so, $\nabla\hat\bnu$ is continuous at $\bx_0$.
\end{proofpart}
\end{proof}

\begin{remark}\label{regremark}
The assumptions placed on an admissible framed curve $(\bd,\bn)$ are only strong enough to generate a $C^1$ isometric immersion that is $C^2$ almost everywhere, rather than $C^2$ on all of $\bar\cD$. To see why this is so, consider a choice of $\cD$ that consists of two squares of different sizes, with one of the sides of the smaller square attached to a side of the larger one. Suppose, also, that the two squares are arranged so that $\partial\cD$ has eight corners. Let $\cL$ denote the line segment in $\cD$ along which the two squares meet. Define an isometry $\bchi$ on $\cD$ by sending one square to a cylinder of radius $R_1$ and the other square to a cylinder of radius $R_2$ in such a way that the referential rulings generated are all parallel to $\cL$. This can be done so that the resulting surface $\cS$ and $\bchi$ are $C^1$. Moreover, it is possible to can check that the framed curve associated with the boundary of $\cS$ is admissible. However, if $R_1\not=R_2$, then $\cS$ will not be $C^2$ on the image of $\cL$, as the mean curvature is not continuous there. Consequently, $\bchi$ is not $C^2$ on $\cL$.

From the proof of Theorem~\ref{fSconstP} it is nevertheless evident that the constructed isometric immersion can only fail to be $C^2$ on a ruling satisfying \emph{\ref{L3}} and that there can be at most a countable number of such rulings. It thus follows that if $\cD$ is convex and that $\partial\cD$ has no corners, the constructed isometry would be $C^2$ on $\cD$, as no such rulings can exist in this case.
\end{remark}

We conlude this section with two results that lead together to the conclusion that condition \emph{\ref{M2}} holds if and only if a single inequality is satisfied.

\begin{proposition}
Condition \emph{\ref{M2}} holds if and only if, for each combination of $(\alpha,\alpha')\in\tilde\cA\times\tilde\cA$, one of the following conditions hold:
\begin{enumerate}[label=\textbf{\textit{M2.\arabic*}}, leftmargin=1em, align=left]
\item \label{M2.1} $\bff(\alpha)\times\bff(\alpha')=\bzero$,

\item \label{M2.2} $\displaystyle\frac{\bff'(\alpha')\cdot(\bc(\alpha')-\bc(\alpha))}{(\bff(\alpha')\times\bN)\cdot\bff(\alpha)}\ge \hat\beta(\alpha),$

\item \label{M2.3} $\displaystyle\frac{\bff'(\alpha')\cdot(\bc(\alpha')-\bc(\alpha))}{(\bff(\alpha')\times\bN)\cdot\bff(\alpha)}\le 0.$
\end{enumerate}

\end{proposition}

\begin{proof}
Assume that \emph{\ref{M2}} holds and fix $\alpha\in\tilde\cA$ and $\alpha'\in\tilde\cA$. The line through $\bc(\alpha)$ with direction $\bff(\alpha)$ is is either parallel or not parallel to the line through $\bc(\alpha')$ with direction $\bff(\alpha')$. If these lines are parallel, then \emph{\ref{M2.1}} holds. If these lines are not parallel, then they intersect and, hence, there exist $\beta\in\Real$ and $\beta'\in\Real$ such that
\beqn
\bc(\alpha)+\beta\bff(\alpha)=\bc(\alpha')+\beta'\bff(\alpha').
\label{464}
\eeqn
From the component of \eqref{464} in the direction of $\bff(\alpha')\times\bN$, which is orthogonal to $\bff(\alpha')$, it follows that
\beqn
\beta=\frac{\bff'(\alpha')\cdot(\bc(\alpha')-\bc(\alpha))}{(\bff(\alpha')\times\bN)\cdot\bff(\alpha)}
\eeqn
and, thus, by \emph{\ref{M2}} that either \emph{\ref{M2.2}} or \emph{\ref{M2.3}} must hold.

Now fix $\alpha\in\tilde\cA$ and $\alpha'\in\tilde\cA$ and assume that one of \emph{\ref{M2.1}}, \emph{\ref{M2.2}}, and \emph{\ref{M2.3}} holds. It is then necessary to show that $\cL=[\bc(\alpha), \bc(\mu(\alpha))]$ and $\cL'=[\bc(\alpha'), \bc(\mu(\alpha'))]$ are the same, meet at an endpoint, or are disjoint.

If \emph{\ref{M2.1}} holds, then $\cL$ and $\cL'$ are parallel. If \emph{\ref{M2}} does not hold, then $\cL$ and $\cL'$ must be different but overlap on an interval of nonzero length. From the definition \ref{mudef} of $\mu$, this is not possible; hence, \eqref{M2} must hold. Next, assume that \ref{M2.2} holds. Notice that this condition is viable only if $\bff(\alpha)$ and $\bff(\alpha')$ are not parallel and, hence, only if $\cL$ and $\cL'$ are not parallel, meaning that the lines containing these rulings must cross. It follows that there are $\beta\in\Real$ and $\beta'\in\Real$ such that 
\beqn
\bc(\alpha)+\beta\bff(\alpha)=\bc(\alpha')+\beta'\bff(\alpha').
\label{467}
\eeqn
From the component of \eqref{467} in the direction of $\bff(\alpha')\times\bN$, it follows that
\beqn
\beta=\frac{\bff'(\alpha')\cdot(\bc(\alpha')-\bc(\alpha))}{(\bff(\alpha')\times\bN)\cdot\bff(\alpha)},
\eeqn
from which it follows that $\beta\ge\hat\beta(\alpha)$ and, thus, that the lines  containing $\cL$ and $\cL'$ must cross outside of $\cD$ or meet on $\partial\cD$. Hence, $\cL$ and $\cL'$ are either disjoint or meet on $\partial\cD$. The proof involving \emph{\ref{M2.3}} is similar and, thus, is omitted.
\end{proof}

The conditions \emph{\ref{M2.1}}, \emph{\ref{M2.2}}, and \emph{\ref{M2.3}} can be succinctly represented by a single inequality.

\begin{proposition}
Given $\alpha\in\tilde\cA$ and $\alpha'\in\tilde\cA$, let $M_i(\alpha,\alpha')$, $i=1,2,3$, be defined by 
\begin{align}
M_1(\alpha,\alpha')&=|\bff(\alpha)\times\bff(\alpha')|^2,
\\[4pt]
M_2(\alpha,\alpha')&=\bff'(\alpha')\cdot(\bc(\alpha')-\bc(\alpha))- \hat\beta(\alpha)(\bff(\alpha')\times\bN)\cdot\bff(\alpha),
\\[6pt]
M_3(\alpha,\alpha')&=\bff'(\alpha')\cdot(\bc(\alpha')-\bc(\alpha)).
\end{align}
Then, one of \emph{\ref{M2.1}}--\emph{\ref{M2.3}} holds if and only if
\beqn\label{M2simp}
M_1(\alpha,\alpha')M_2(\alpha,\alpha')M_3(\alpha,\alpha')\ge 0
\eeqn
for each $\alpha\in\tilde\cA$ and $\alpha'\in\tilde\cA$.
\end{proposition}

\begin{proof}
Fix $\alpha\in\tilde\cA$ and $\alpha'\in\tilde\cA$. The proof is split into cases depending on whether $(\bff(\alpha')\times\bN)\cdot\bff(\alpha)=0$, $(\bff(\alpha')\times\bN)\cdot\bff(\alpha)>0$, or $(\bff(\alpha')\times\bN)\cdot\bff(\alpha)<0$. First, if $(\bff(\alpha')\times\bN)\cdot\bff(\alpha)=0$, then, since $(\bff(\alpha')\times\bN)\cdot\bff(\alpha)=\pm|\bff(\alpha')\times\bff(\alpha)|$, \emph{\ref{M2.1}} and \eqref{M2simp} hold.

Next, for $(\bff(\alpha')\times\bN)\cdot\bff(\alpha)>0$, suppose that one of \emph{\ref{M2.2}}--\emph{\ref{M2.3}} hold. If, in particular, \emph{\ref{M2.2}} holds, then $M_2(\alpha,\alpha')\ge 0$ and $M_3(\alpha,\alpha')>0$, whereby \eqref{M2simp} holds. If, alternatively, \emph{\ref{M2.3}} holds, then $M_3(\alpha,\alpha') \le 0$ and $M_2(\alpha,\alpha')<0$, whereby \eqref{M2simp} still holds. Suppose, conversely, that \eqref{M2simp} holds. Since $M_1(\alpha,\alpha')>0$, it follows from \eqref{M2simp} that
\beqn
M_2(\alpha,\alpha')M_3(\alpha,\alpha')\ge 0.
\label{472}
\eeqn
Referring to \eqref{472}, $M_2(\alpha,\alpha')$ and $M_3(\alpha,\alpha')$ are either both positive or both negative. If they are both positive, then \emph{\ref{M2.2}} holds; if they are both negative, then \emph{\ref{M2.3}} holds.

The proof for the case $(\bff(\alpha')\times\bN)\cdot\bff(\alpha)<0$ is completely analogous to that for $(\bff(\alpha')\times\bN)\cdot\bff(\alpha)>0$ and is therefore omitted.
\end{proof}

\section{Summary and discussion}
\label{sectcon}

We studied the class of $C^2$ isometric immersions defined on the closure $\bar\cD$ of a simply connected, bounded, planar region $\cD$ whose boundary may include finitely many corners and intervals of vanishing curvature. The image surface $\cS=\bchi(\cD)$ associated with such an immersion $\bchi$ admits a natural decomposition into curved and planar regions, with the curved regions described by ruled surfaces and the planar regions arising in zones where the mean curvature vanishes. We associated to each immersion a framed curve $(\bd,\bn)$ defined along the boundary $\partial\cS$, where $\bd$ is a parametrization of $\partial\cS$ and $\bn$ is a unit normal to $\cS$ restricted to this curve. We found that the framed curve must satisfy regularity conditions, be compatible with $\partial\cD$, and be associated with rulings that satisfy specific geometric constraints. Starting from any framed curve with finite total bending energy $E$ satisfying these conditions, we constructed a $C^1$ isometric immersion $\bar\cD$ that is $C^2$ almost everywhere on $\cD$.

The requirement that the bending energy $E$ be finite is essential to the construction of the isometric immersion. Specifically, if $E$ is finite, then $\bn'$ vanishes precisely on open intervals along the boundary $\partial\cS$ that boarder flat regions of the surface $\cS=\bchi(\cD)$ and is nonzero only on those associated with curved, ruled portions of $\cS$. This relationship is used in defining the set $\hat\cA$ and permits a locally invertible ruled parameterization to be established on the curved regions where the mean curvature of $\cS$ is nonzero. The role of the finite energy condition in this analysis is consistent with both the physical modeling of unstretchable elastic surfaces and established variational frameworks involving curvature-dependent energy densities.

Although the immersions obtained by our construction do not exhibit global $C^2$ regularity, they are globally $C^1$, including the boundary, and $C^2$ almost everywhere. This outcome is consistent with expectations based on known examples, including Sadowsky's piecewise smooth construction of a developable M\"obius band. The restriction to $C^2$ immersions in our analysis is not overly restrictive due to the density of smooth maps in the Sobolev space $W^{2,2}$, the natural setting for isometric immersions with finite bending energy. We recognize the potential for extending the present framework to that setting, where it might be useful in characterizing such immersions in terms of boundary-framed data.

A distinguishing feature of our formulation is the representation of isometric immersions in terms of framed curves along the boundary of the image surface. The resulting framework provides a clear and structured approach for studying the geometry and energetics of developable surfaces under boundary constraints. Beyond its implications for differential geometry and variational theory, the framework developed here may also inform applications in the mechanics of thin structures, where accurate control of regularity and curvature is essential for predictive modeling.

\section*{Acknowledgements}

E.F. gratefully acknowledges support from the Okinawa Institute of Science and Technology Graduate University with subsidy funding from the Cabinet Office, Government of Japan.

%\section*{Data availability statement}
%
%There is no data associated with this work.
%
%\section*{Conflict of interest statement}
%
%The authors declare that they have no conflict of interest.

\bibliographystyle{acm}
\bibliography{isometricdef} 

\end{document}